\documentclass{amsart}
\usepackage{amsmath}
\usepackage{amsfonts}
\usepackage{graphics}
\usepackage{epsfig} 
\usepackage{amssymb}
\usepackage{amscd}
\usepackage[all]{xy}
\usepackage{latexsym}
\usepackage{graphicx}
\usepackage{multirow}
\usepackage{geometry}

\newtheorem{thm}{Theorem}[section]
\newtheorem{cor}[thm]{Corollary}
\newtheorem{lem}[thm]{Lemma}
\newtheorem{prop}[thm]{Proposition}
\newtheorem{conj}[thm]{Conjecture}
\theoremstyle{definition}
\newtheorem{Def}[thm]{Definition}
\newtheorem{rem}[thm]{Remark}
\newtheorem*{ack}{Acknowledgement}

\numberwithin{equation}{section}
\numberwithin{figure}{section}

\def\rchi{{\hbox{\raise1.5pt\hbox{$\chi$}}}}
\def\Aut{{\text{\rm{Aut}}}}

\def\tensor{\otimes}

\def\const{{\text{\rm{const}}}}

\def\a{\alpha}
\def\b{\beta}
\def\lam{\lambda}
\def\Lam{\Lambda}

\def\Res{{\text{\rm{Res}}}}

\newcommand{\bP}{{\mathbb{P}}}
\newcommand{\bC}{{\mathbb{C}}}

\newcommand{\bN}{{\mathbb{N}}}
\newcommand{\bE}{{\mathbb{E}}}

\newcommand{\bQ}{{\mathbb{Q}}}
\newcommand{\bR}{{\mathbb{R}}}

\newcommand{\bZ}{{\mathbb{Z}}}

\newcommand{\cM}{{\mathcal{M}}}

\newcommand{\cH}{{\mathcal{H}}}

\newcommand{\la}{{\langle}}
\newcommand{\ra}{{\rangle}}
\newcommand{\half}{{\frac{1}{2}}}

\newcommand{\hxi}{{\hat{\xi}}}
\newcommand{\hatH}{{\widehat{\cH}}}

\begin{document}
\large
\setcounter{section}{0}

\title[{L}aplace transform of {H}urwitz numbers]{The
{L}aplace transform  of the cut-and-join equation
and the {B}ouchard-{M}ari\~no conjecture on
{H}urwitz numbers}
\author[B.~Eynard]{Bertrand Eynard}  
\address{
Service de Physique Th\'eorique de Saclay\\
Gif-sur-Yvette\\
F-91191 Cedex, France}
\email{bertrand.eynard@cea.fr}
\author[M.~Mulase]{Motohico Mulase}  
\address{
Department of Mathematics\\
University of California\\
Davis, CA 95616--8633}
\email{mulase@math.ucdavis.edu}
\author[B.~Safnuk]{Bradley Safnuk}  
\address{
Department of Mathematics\\
Central Michigan University\\
Mount Pleasant, MI 48859}
\email{brad.safnuk@cmich.edu}

\begin{abstract}
We calculate the Laplace transform of the
cut-and-join equation of Goulden, Jackson and Vakil.
The result is a polynomial equation that has the
topological structure identical to the 
Mirzakhani recursion formula
for the Weil-Petersson volume of the moduli space
of bordered hyperbolic surfaces.
We find that the direct image of this
Laplace transformed equation via the inverse of  the
Lambert W-function is the topological recursion 
formula for Hurwitz numbers
conjectured by Bouchard and Mari\~no using 
topological string theory.
\end{abstract}


\thanks{Saclay preprint number: IPHT T09/101}

\maketitle

\allowdisplaybreaks

\tableofcontents

\section{Introduction}
\label{sect:intro}

The purpose of this paper is to 
give a proof of the
Bouchard-Mari\~no conjecture \cite{BM} on  Hurwitz numbers
using
the Laplace transform of the 
celebrated \emph{cut-and-join equation} 
of Goulden, Jackson, and Vakil \cite{GJ, V}.
The cut-and-join equation, which seems to be
essentially known to Hurwitz \cite{H},
 expresses the Hurwitz number of a
given genus and profile (partition) in terms of those corresponding to 
profiles modified by either \emph{cutting} a part into two pieces
or \emph{joining} two parts into one. 
This equation holds for an arbitrary partition $\mu$.
We calculate the Laplace transform of this equation 
with $\mu$ as the summation variable. The result is a
\emph{polynomial equation} \cite{MZ}.

A \emph{Hurwitz} cover is a holomorphic mapping 
$f:X\rightarrow \mathbb{P}^1$ from 
a connected nonsingular projective 
algebraic curve $X$ of genus $g$ to 
the projective line $\mathbb{P}^1$ with only simple 
ramifications except for $\infty \in\mathbb{P}^1$. 
Such a cover is further refined by specifying its
\emph{profile}, which is a partition $\mu = (\mu_1\ge \mu_2 \ge
\cdots \ge \mu_\ell >0)$ of  the degree of the covering $\deg f = |\mu| =
\mu_1 + \cdots + \mu_\ell$.  The length  $\ell(\mu) = \ell$
 of this partition
is the number of points in the inverse image  $f^{-1}(\infty) 
= \{p_1,\dots,p_\ell\}$ of 
$\infty$. 
Each part $\mu_i$ gives a local description of the map $f$,
which is given by $z\longmapsto z^{\mu_i}$ in terms of a local
coordinate $z$ of $X$ around $p_i$. 
The number 
$h_{g,\mu}$ of topological
types of Hurwitz covers of given genus $g$ and profile 
$\mu$, counted with the weight factor $1/|\Aut f|$, is the \emph{Hurwitz
number} we shall deal with in this paper. 
A remarkable formula due to 
Ekedahl, Lando, Shapiro and Vainshtein
\cite{ELSV, GV1, OP1} 
relates Hurwitz numbers and Gromov-Witten
invariants.  For genus
$g\ge 0$ and a partition $\mu$ subject to the stability condition
 $2g - 2 +\ell(\mu) >0$, the ELSV formula states that
\begin{equation}
\label{eq:ELSV}
	h_{g,\mu} = \frac{\big( 2g-2+\ell (\mu)+|\mu|\big)!}{|\Aut( \mu)|}\;  \prod_{i=1}^{\ell(\mu)}\frac{\mu_i ^{\mu_i}}{\mu_i!}\int_{\overline{\mathcal{M}}_{g,\ell(\mu)}} \frac{\Lambda_g^{\vee}(1)}{\prod_{i=1}^{\ell(\mu)}\big( 1-\mu_i \psi_i\big)},
\end{equation}
where $\overline{\mathcal{M}}_{g,\ell}$ is 
the Deligne-Mumford moduli stack of stable algebraic
curves of genus $g$ with $\ell$ distinct marked points, 
$\Lambda_g^{\vee}(1)=1-c_{1}(\bE)+\cdots +(-1)^g
c_{g}(\bE)$ is the alternating sum of
Chern classes of the Hodge bundle $\bE$  
 on $\overline{\mathcal{M}}_{g,\ell}$, 
$\psi_i$ is the $i$-th tautological cotangent class, 
and $\Aut(\mu)$ denotes the group of permutations of 
equal parts of the partition $\mu$. The linear Hodge integrals
are the rational numbers defined by 
\begin{equation*}
\la \tau_{n_1}\cdots \tau_{n_\ell}c_{j}(\bE)\ra= \int_{\overline{\mathcal{M}}_{g,\ell}}\psi_1^{n_1}\cdots\psi_\ell ^{n_\ell}c_j(\bE),
\end{equation*}
which is $0$ unless $n_1+\cdots+n_\ell +j=3g-3+\ell$.
To present our main theorem, 
let us  introduce a series of polynomials $\hxi_n(t)$ of
degree $2n+1$ in $t$ for $n\ge 0$ by the
recursion formula
\begin{equation*}
\hxi_{n}(t) = t^2(t-1)\frac{d}{dt}\hat{\xi}_{n-1}(t)
\end{equation*}
with the initial condition
$\hxi_{0}(t) =t-1$. 
This differential operator appears in \cite{GJV2}.
The   Laplace transform of the cut-and-join
equation gives the following formula.

\begin{thm}[\cite{MZ}]
\label{thm:intro-main}
Linear Hodge integrals satisfy  recursion relations given as
 a series of 
equations of symmetric polynomials in 
$\ell$ variables $t_1,\dots,t_\ell$:
\begin{multline}
\label{eq:main}
\sum_{n_L}\la \tau_{n_L}\Lam_g ^\vee(1)\ra_{g,\ell}
\left(
 (2g-2+\ell)\hxi_{n_L}(t_L)
+
\sum_{i=1} ^\ell
\frac{1}{t_i}\;
 \hxi_{n_i+1}(t_i)\hxi_{L\setminus\{i\}}(
t_{L\setminus\{i\}})
\right)
\\
=
\sum_{i<j}
\sum_{m,n_{L\setminus\{i,j\}}}
\la \tau_m\tau_{n_{L\setminus\{i,j\}}}
\Lam_{g} ^\vee (1)\ra_{g,\ell-1}
\hxi_{n_{L\setminus\{i,j\}}}(t_{L\setminus\{i,j\}})
\frac{\hxi_{m+1}(t_i) \hxi_0(t_j) t_i ^2 
-
\hxi_{m+1}(t_j) \hxi_0(t_i) t_j ^2}{t_i-t_j}
\\
+
\half\sum_{i=1}^\ell
\sum_{n_{L\setminus\{i\}}}\sum_{a,b}
\Bigg(
\la \tau_a\tau_b\tau_{n_{L\setminus\{i\}}}\Lam_{g-1} ^\vee
(1)\ra_{g-1,\ell+1}\\
+
\sum_{\substack{g_1+g_2 = g\\
I\sqcup J = L\setminus\{i\}}} ^{\rm{stable}}
\la \tau_a\tau_{n_{I}}\Lam_{g_1} ^\vee
(1)\ra_{g_1,|I|+1}
\la \tau_b\tau_{n_{J}}\Lam_{g_2} ^\vee
(1)\ra_{g_2,|J|+1}
\Bigg)
\hxi_{a+1}(t_i)\hxi_{b+1}(t_i)\hxi_{n_{L\setminus\{i\}}}(
t_{L\setminus\{i\}}),
\end{multline}
where $L=\{1,2\dots,\ell\}$ is an index set, and for a subset 
$I\subset L$, we denote
$$
t_I = (t_i)_{i\in I},\quad
	 n_I = \{\, n_i\,|\, i\in I\,\},\quad 
	\tau_{n_I} = \prod_{i\in I}\tau_{n_i},
	\quad
		\hxi_{n_I}(t_I) = \prod_{i\in I}
		\hxi_{n_i}(t_i).
$$
The last summation in the formula
 is taken over all partitions of $g$ and 
decompositions  of $L$ into 
disjoint subsets $I\sqcup J=  L$ subject to the stability 
condition $2g_1 - 1 + |I| >0$ and
$2 g_2 -1 +|J| >0$.
\end{thm}

\begin{rem}
We note a similarity of the above formula
 and the Mirzakhani
recursion formula for the Weil-Petersson volume of the
moduli space of bordered hyperbolic surfaces of genus $g$
with $\ell$ closed geodesic boundaries \cite{Mir1, Mir2}.
\begin{enumerate}

\item
There is no a priori reason for the 
Laplace transform to be
a \emph{polynomial} equation.

\item 
The above formula  is a \emph{topological
recursion}. 
For an algebraic curve of
 genus $g\ge 0$ and $\ell \ge 1$  distinct
marked points on it, the absolute value of the 
Euler characteristic  of the $\ell$-punctured 
Riemann surface, $2g - 2 +\ell$,  defines a \emph{complexity} of 
the moduli space $\overline{\cM}_{g,\ell}$. 
Eqn.(\ref{eq:main}) gives an effective
method of calculating the linear Hodge integrals of
complexity $n>0$ from those with complexity
$n-1$. 

\item
When we restrict (\ref{eq:main}) to the 
homogeneous highest degree terms, 
the equation reduces to the  Witten-Kontsevich
theorem of
 $\psi$-class intersections \cite{DVV,K1992,W1991}.

\end{enumerate}
\end{rem}

Let us explain the background of our work.
Independent of the recent geometric and combinatorial works
\cite{GJ,Mir1,Mir2,V}, a theory of topological 
recursions has been developed in the matrix model/random 
matrix theory community \cite{E2004, EO1}. Its 
culmination is the  topological
 recursion formula established in
\cite{EO1}. There are three ingredients in this theory:
the Cauchy differentiation kernel (which is referred to
as  ``the Bergman
Kernel'' in \cite{BM, EO1}) of an
 analytic curve $C\subset \bC ^2$ 
 in the $xy$-plane called a \emph{spectral
curve},
 the standard holomorphic symplectic 
structure on $\mathbb{C}^2$, 
and the ramification 
behavior of the projection
$\pi:C\rightarrow \bC$ of the spectral curve to the $x$-axis. 
When $C$ is  hyperelliptic whose ramification 
points are all real, the topological recursion  solves 
$1$-Hermitian matrix models for the potential function
that determines the spectral curve. It  means  that the 
formula recursively 
computes all $n$-point correlation functions
of the resolvent of  random Hermitian matrices of an arbitrary size. 
By   choosing  a particular
 spectral curve of  genus $0$, the topological
recursion \cite{E2007, EO1, EO2} recovers the Virasoro constraint conditions for
the $\psi$-class intersection numbers 
$\la \tau_{n_1} \cdots \tau_{n_\ell}\ra$ due to Witten \cite{W1991}
and Kontsevich \cite{K1992}, and 
the  mixed intersection numbers
$\la \kappa_1 ^{m_1} \kappa_2 ^{m_2} \cdots  \tau_{n_1} \cdots \tau_{n_\ell}\ra$ due to Mulase-Safnuk \cite{MS} and
Liu-Xu \cite{LX}.
Based on the work by Mari\~no \cite{Mar2006}
and Bouchard, Klemm, Mari\~no and 
Pasquetti 
\cite{BKMP}
on {\em remodeling} the B-model
topological string theory on the mirror curve of a
toric Calabi-Yau $3$-fold, 
Bouchard and Mari\~no \cite {BM} conjecture that when one uses the \emph{Lambert curve}
\begin{equation}
\label{eq:Lambert-xy}
C = \{(x,y)\;|\; x = y e^{-y}\} \subset \bC ^* \times \bC ^*
\end{equation}
as the spectral curve, 
the   topological  recursion formula of Eynard and Orantin
should compute the generating functions 
\begin{equation}
\begin{aligned}
\label{eq:Hgell}
H_{g,\ell} (x_1,\dots,x_\ell) 
&= \sum_{\mu : \ell(\mu) = \ell}
\frac{\mu_1\mu_2\cdots\mu_\ell}{(2g - 2 +\ell +|\mu|)!}\; h_{g,\mu}
\sum_{\sigma\in S_\ell}\prod_{i=1} ^\ell 
x_{\sigma(i)}  ^{\mu_i -1}
\\
&=
 \sum_{n_1+\cdots+n_\ell \leq 3g-3+\ell} \la \tau_{n_1}\cdots \tau_{n_\ell}\Lambda_g^{\vee}(1)\ra\,\,\prod_{i=1}^{\ell}\sum_{\mu_i=1}^{\infty}\frac{\mu_i^{\mu_i+1+n_i}}{\mu_i!}x_i^{\mu_i-1}
\end{aligned}
\end{equation}
of Hurwitz numbers 
 for all $g\ge 0$ and $\ell>0$.
Here the sum in the first line 
is taken over all partitions $\mu$ of  length 
$\ell$, and $S_\ell$ is the symmetric group of $\ell$ letters.

\begin{figure}[htb]
\centerline{\epsfig{file=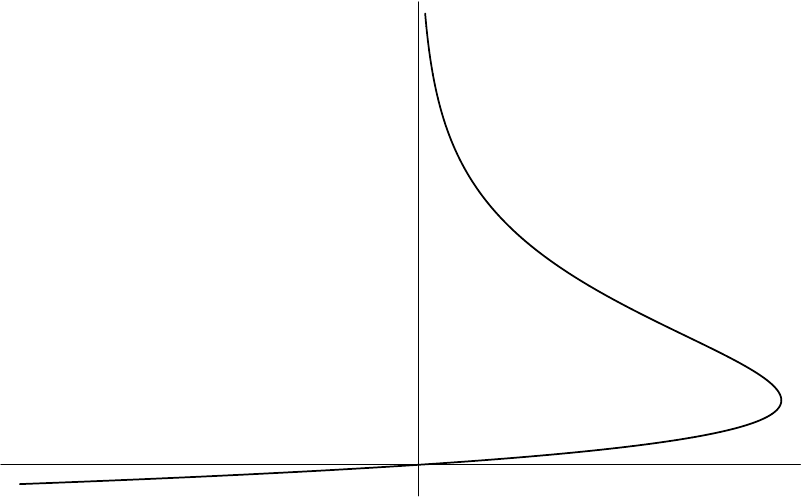, width=2in}}
\caption{The Lambert curve $C\subset \bC ^* \times \bC ^*$ 
defined by  $x=y e^{-y}$.}
\label{fig:curve}
\end{figure}

Our discovery of this paper is that the Laplace transform
of the combinatorics, the cut-and-join equation in our case, explains
the role of the Lambert curve, the ramification behavior of the 
projection
$\pi:C\rightarrow \bC^*$, the Cauchy differentiation 
kernel  on $C$,
and residue calculations
that appear in the theory of topological recursion. 
As a consequence of this explanation, we obtain 
a proof of the Bouchard-Mari\~no conjecture \cite{BM}.
For this purpose, it is essential to use
 a different parametrization of the Lambert curve:
 $$
x = e^{-(w+1)}\qquad{\text{and}}\qquad
y=\frac{t-1}{t} .
$$
 The coordinate $w$ is the
parameter  of the Laplace transformation, which changes
a function 
 in positive integers to a complex analytic function in $w$. 
Recall the Stirling expansion 
$$
e^{-k}\frac{k^{k+n}}{k!}\sim \frac{1}{\sqrt{2\pi}}\;k^{n-\half},
$$
which makes its Laplace transform a function of $\sqrt{w}$. 
Note that the $x$-projection $\pi$ of the 
Lambert curve (\ref{eq:Lambert-xy}) is locally a double-sheeted covering
around its unique critical point $(x=e^{-1},y=1)$.
Therefore, the Laplace transform of the ELSV formula
(\ref{eq:ELSV})  naturally lives on
the  Lambert curve $C$ rather than on the $w$-coordinate plane.
Note that $C$ is an analytic curve of genus $0$ and $t$ is
its global coordinate.  The point at infinity $t=\infty$
is the ramification point of the projection $\pi$.   
In terms of these coordinates, the Laplace transform 
of the ELSV formula becomes a \emph{polynomial} in $t$-variables.

The proof of the Bouchard-Mari\~no conjecture 
 is established as
follows. 
A topological recursion of \cite{EO1} is always 
given as a residue formula of symmetric differential forms
 on the spectral curve. The Laplace-transformed cut-and-join equation (\ref{eq:main})
 is an equation among \emph{primitives} of differential
 forms. We first take the exterior differential of
 this equation. 
 We then analyze the role of the residue calculation
 in the theory of topological recursion \cite{BM, EO1},
 and find that it is equivalent to evaluating the
 form at $q\in C$ and its conjugate point $\bar{q}\in C$
 with respect to the local Galois covering $\pi:C\rightarrow \bC$
 near its critical point.
This means all residue calculations are replaced by
 an algebraic operation of taking the direct image
 of the differential form via the projection $\pi$.
 We find that the direct image of (\ref{eq:main})  then 
 becomes identical to the conjectured formula
(\ref{eq:BMalgebraic}).

 \begin{thm}[The Bouchard-Mari\~no Conjecture]
 \label{thm:BMconj}
 The linear Hodge integrals satisfy
 exactly the same 
 topological recursion formula discovered in \cite{EO1}:
 \begin{multline}
\label{eq:BMalgebraic}
\sum_{n,n_L}\la \tau_n\tau_{n_L} \Lambda_g ^\vee (1) \ra_{g,\ell+1}
 d \hxi_n(t) \tensor
d \hxi_{n_L}(t_L)\\
=
\sum_{i=1} ^\ell
\sum_{m,n_{L\setminus\{i\}}}\la \tau_m\tau_{n_{L\setminus\{i\}}} \Lambda_{g} ^\vee (1) \ra_{g,\ell}
P_{m}(t,t_i)
dt\tensor dt_i\tensor
d \hxi_{n_{L\setminus\{i\}}}(t_{L\setminus\{i\}})\\
+
\Bigg(
\sum_{a,b,n_L}\la \tau_a\tau_b\tau_{n_L} \Lambda_{g-1} ^\vee (1) \ra_{g-1,\ell+2}
\\
+
\sum_{\substack{g_1+g_2=g\\
I\sqcup J=L}} ^{\rm{stable}}
\sum_{\substack{a,n_I\\b,n_J}}\la \tau_a\tau_{n_I}
\Lambda_{g_1} ^\vee (1) \ra_{g_1,|I|+1}\la \tau_b\tau_{n_J}
\Lambda_{g_2} ^\vee (1) \ra_{g_2,|J|+1}
\Bigg)
P_{a,b}(t)dt\tensor  d \hxi_{n_L}(t_L),
\end{multline}
where
$$
	d\hxi_{n_I}(t_I) = \bigotimes_{i\in I}
		\frac{d}{dt_i}\hxi_{n_i}(t_i)dt_i.
		$$
The functions $P_{a,b}(t)$ and $P_n(t,t_i)$ 
  are defined by 
taking the \emph{polynomial} part of the expressions
\begin{align*}
&P_{a,b}(t)dt =\half \left[ \frac{ts(t)}{t-s(t)}\;
\frac{dt}{t^2(t-1)}\Big(
\hat{\xi}_{a+1}(t)\hat{\xi}_{b+1}\big(s(t)\big)
+\hat{\xi}_{a+1}\big(s(t)\big)\hat{\xi}_{b+1}(t)
\Big)
\right]_+,\\
&P_n(t,t_i) dt\tensor dt_i=d_{t_i}\left[\frac{ts(t)}{t-s(t)}
\left(\frac{\hat{\xi}_{n+1} (t)ds(t)}{s(t)-t_i} 
+ \frac{\hat{\xi}_{n+1}\big(s(t)\big)dt}
{t-t_i}\right)\right]_+,
\end{align*}
where $s(t)$ is the deck-tranformation of the
projection $\pi:C\rightarrow \bC^*$ around 
its critical point $\infty$. 
\end{thm}

\noindent
The relation between the cut-and-join formula, 
(\ref{eq:main}) and 
(\ref{eq:BMalgebraic}) is the following:
\begin{equation*}
\begin{CD}
{\substack{\text{Cut-and-Join}\\
{\text{Equation}}}}
@>>{\text{Laplace Transform}}>
{\substack{\text{Polynomial Equation}\\
{\text{on Primitives (\ref{eq:main})}}}}
@>>{\text{Direct Image}}>
{\substack{\text{Topological Recursion on}\\
{\text{Differential Forms (\ref{eq:BMalgebraic})}}}}
\\
@VVV @VVV @VVV
\\
\{{\text{Partitions}}\} @>>>
{\text{Lambert Curve}}
@>{\text{Galois Cover}}>>\bC
\end{CD}
\end{equation*}

Mathematics of the 
topological recursion  and its
geometric realization 
includes still many more mysteries
\cite{BKMP, DV2007, EO1, Mar2006}. 
Among them is a relation to integrable systems such as
the Kadomtsev-Petviashvili equations \cite{EO1}. 
In recent years these equations have played
 an essential role in the study of Hurwitz numbers
 \cite{Kazarian, KazarianLando, MirMor, O2000, OP1, OP2}. 
 Since the aim of this paper is to give a proof of the 
 Bouchard-Mari\~no conjecture and to give a 
 geometric interpretation of the topological 
 recursion for the Hurwitz case, we do not address 
 this relation  here. Since we relate the nature of the topological 
 recursion and combinatorics  
 by the Laplace transform, it is reasonable to ask:
 what is the \emph{inverse} Laplace transform of 
 the topological recursion in general? This  question
 relates the Laplace transformation and the mirror symmetry.
 These are interesting topics to be further explored.

 It is possible to prove the Bouchard-Mari\~no formula
 without appealing to the cut-and-join equation. Indeed, 
 a matrix integral expression of the generating 
 function of Hurwitz numbers  has been  recently discovered in
 \cite{BEMS}, and its spectral curve is identified 
 as the Lambert curve.  As a consequence, 
 the symplectic 
 invariant theory of matrix models  \cite{E2004, E2008} 
 is directly applicable to Hurwitz theory. The discovery of
 \cite{BEMS} is that the derivatives of the 
 symplectic invariants of the
 Lambert curve give $H_{g,\ell}(x_1,\dots,x_\ell)$ of (\ref{eq:Hgell}).
The topological recursion 
 formula of Bouchard and Mari\~no
 then automatically follows.  A deeper understanding of 
 the interplay between these two
 totally different techniques
  is desirable.

 Although our statement is simple and clear, 
 technical details are quite involved.
 We have decided to provide all key details in this paper,
believing that some of the analysis may be useful
 for further study of more general topological recursions.
 This explains the length of the current paper
 in  the sections dealing with complex analysis
 and the Laplace transforms.

The paper is organized as follows. We start with 
identifying the generating function (\ref{eq:Hgell})
as the Laplace transform of the ELSV formula (\ref{eq:ELSV}) in
Section~\ref{sect:ELSV}. We then calculate the 
Laplace transform of the cut-and-join equation
in Section~\ref{sect:CAJ} following
\cite{MZ}, and present the key idea of the proof of
Theorem~\ref{thm:intro-main}. 
In
Section~\ref{sect:BM}
we give the statement of the 
Bouchard and Mari\~no conjecture \cite{BM}. 
We calculate the residues appearing in the 
topological recursion formula in Section~\ref{sect:residue}
for the case of Hurwitz generating functions. The topological
recursion becomes the algebraic relation as presented 
in Theorem~\ref{thm:BMconj}. 
  In 
Section~\ref{sect:Laplace} we prove technical statements
necessary for reducing (\ref{eq:main}) to (\ref{eq:BMalgebraic})
as a Galois average.
The final 
Section~\ref{sect:proof} is devoted to proving
the Bouchard-Mari\~no conjecture.

As an effective recursion, (\ref{eq:main}) and 
(\ref{eq:BMalgebraic}) calculate
linear Hodge integrals, and hence Hurwitz numbers
through the ELSV formula.
A computation is performed by  Michael Reinhard,
an undergraduate
student of UC Berkeley. 
We reproduce some  of his tables
at the end of the paper.

\begin{ack}
Our special thanks are due to 
Vincent Bouchard 
 for numerous discussions and tireless explanations of the
 recursion formulas and the \emph{Remodeling} theory. We thank
 the Institute for the Physics 
and Mathematics of the Universe,
the Osaka City University Advanced Mathematical 
Institute, and the American Institute
of Mathematics for their hospitality during our stay and for
providing us with the opportunity 
of collaboration. 
Without their assistance, this collaboration would have
never been started. We also thank Ga\"etan Borot, 
 Yon-Seo Kim, 
Chiu-Chu Melissa  Liu, Kefeng Liu, 
Marcos Mari\~no, Nicolas Orantin, and Hao Xu
for discussions, and Michael Reinhard for his permission 
to reproduce his computation of linear Hodge integrals and
Hurwitz numbers in this paper. 
During the period of preparation of this work,
B.E.'s research was supported in part 
by  the ANR project Grandes Matrices Al\'eatoires 
ANR-08-BLAN-0311-01,
the European Science Foundation through the Misgam program,
and the Quebec government with the FQRNT;
M.M.\ received financial support from  the NSF, Kyoto 
University, T\^ohoku University, KIAS in Seoul,
and the University of Salamanca;
and B.S.\ received support from 
IPMU. 
\end{ack}

\section{The Laplace transform of the
ELSV formula}
\label{sect:ELSV}

In this section we calculate the Laplace transform of
the ELSV formula as a function in partitions $\mu$.
The result is a symmetric polynomial on the Lambert curve
(\ref{eq:Lambert-xy}).

 A \emph{Hurwitz cover} is a smooth morphism
$f:X \rightarrow \bP ^1$ of a connected nonsingular 
projective algebraic curve $X$ of genus $g$
 to $\bP ^1$ that has only simple ramifications 
except for the point at infinity $\infty\in \bP ^1$. Let 
$f^{-1}(\infty) = \{p_1,\dots,p_\ell\}$. Then the induced morphism
of the formal completion 
$\hat{f}: \hat{X}_{p_i}\rightarrow 
\hat{\bP} ^1 _\infty$
is given by $z\longmapsto z^{\mu_i}$ with a positive integer
$\mu_i$ in terms of a formal parameter $z$ around $p_i\in X$. 
We rearrange integers $\mu_i$'s so that 
$
\mu = (\mu_1\ge \mu_2 \ge \cdots \ge \mu_\ell >0)
$
is a partition of $\deg f = |\mu| = \mu_1 + \cdots + \mu_\ell$ 
 of length $\ell(\mu) = \ell$. We call $f$ a Hurwitz cover
of genus $g$ and \emph{profile} $\mu$. A holomorphic automorphism
of a Hurwitz cover is an automorphism $\phi$ of $X$ 
that preserves $f$:
\begin{equation*}
		\xymatrix{
		X\ar[dr]_f  \ar[rr]^\phi _\sim &&X\ar[dl]^f\\
		&\bP ^1&.
		}
\end{equation*} 
Two Hurwitz covers $f_1:X_1 \rightarrow \bP ^1$
and $f_2:X_2 \rightarrow \bP ^1$ are topologically equivalent 
if there is a homeomorphism $h:X_1\rightarrow X_2$ such that
\begin{equation*}
		\xymatrix{
		X_1\ar[dr]_{f_1}  \ar[rr]^h  &&X_2\ar[dl]^{f_2}\\
		&\bP ^1&.
		}
\end{equation*} 
The \emph{Hurwitz number} of type $(g,\mu)$ is defined by
\begin{equation*}
h_{g,\mu} = \sum_{[f]} \frac{1}{|\Aut f|},
\end{equation*}
where the sum is taken over  all topologically equivalent classes
of Hurwitz covers of a given genus $g$ and profile $\mu$. 
Although $h_{g,\mu}$ appears to be a rational number,
it is indeed an integer for most of the cases because 
$f$ has usually no non-trivial automorphisms.
The celebrated ELSV formula \cite{ELSV, GV1, OP1} relates 
Hurwitz numbers and linear Hodge integrals 
on the Deligne-Mumford moduli stack 
$\overline{\mathcal{M}}_{g,\ell}$
consisting of stable algebraic curves of genus $g$ with $\ell$ distinct
nonsingular marked points. 
Denote by $\pi_{g,\ell}:\overline{\mathcal{M}}_{g,\ell+1}
\rightarrow \overline{\mathcal{M}}_{g,\ell}$
the natural projection and by $\omega_{\pi_{g,\ell}}$ the relative
dualizing sheaf of the universal curve $\pi_{g,\ell}$. 
The \emph{Hodge} bundle $\bE$ 
on $\overline{\mathcal{M}}_{g,\ell}$
is defined by $\bE = (\pi_{g,\ell})_* \omega_{\pi_{g,\ell}}$,
and  the $\lambda$-classes
are  the Chern classes of the Hodge bundle:
$$
\lambda_i = c_i(\bE)
\in H^{2i}(\overline{\mathcal{M}}_{g,\ell}, \bQ).
$$
Let $\sigma_i:\overline{\mathcal{M}}_{g,\ell}
\rightarrow \overline{\mathcal{M}}_{g,\ell+1}$
be the $i$-th tautological section of $\pi$, and put
$\mathcal{L}_i = \sigma_i ^*(\omega_{\pi_{g,\ell}})$. 
The $\psi$-classes are defined by
$$
\psi_i = c_1(\mathcal{L}_i) \in 
H^{2}(\overline{\mathcal{M}}_{g,\ell}, \bQ).
$$
The ELSV formula then reads
\begin{equation*}
	h_{g,\mu} = \frac{r!}{|\Aut( \mu)|}\;  \prod_{i=1}^{\ell(\mu)}\frac{\mu_i ^{\mu_i}}{\mu_i!}\int_{\overline{\mathcal{M}}_{g,\ell(\mu)}} \frac{\Lambda_g^{\vee}(1)}{\prod_{i=1}^{\ell(\mu)}\big( 1-\mu_i \psi_i\big)},
\end{equation*}
where $r =r(g,\mu)= 2g-2+\ell (\mu)+|\mu|$ is the number of 
simple ramification points of $f$.

The Deligne-Mumford stack $\overline{\mathcal{M}}_{g,\ell}$
is defined as the moduli space of \emph{stable} curves satisfying the
stability condition
$2-2g-\ell <0$.  However, Hurwitz numbers 
are well defined for \emph{unstable} geometries
 $(g,\ell) = (0,1)$ and $(0,2)$. 
It is an elementary exercise to 
show that
\begin{equation*}
h_{0,k} = k^{k-3}\qquad \text{and}\qquad
h_{0,(\mu_1,\mu_2)} = \frac{(\mu_1+\mu_2)!}{\mu_1+\mu_2}\cdot
\frac{\mu_1^{\mu_1}}{\mu_1!}\cdot \frac{\mu_2^{\mu_2}}{\mu_2!}.
\end{equation*}
The ELSV formula remains true for unstable cases
by \emph{defining}
\begin{align}
\label{eq:01Hodge}
&\int_{\overline{\cM}_{0,1}} \frac{\Lambda_0 ^\vee (1)}{1-k\psi}
=\frac{1}{k^2},\\
\label{eq:02Hodge}
&\int_{\overline{\cM}_{0,2}} 
\frac{\Lambda_0 ^\vee (1)}{(1-\mu_1\psi_1)(1-\mu_2\psi_2)}
=\frac{1}{\mu_1+\mu_2}.
\end{align}

Now fix an $\ell\ge 1$, and consider a partition
$\mu$ of length $\ell$ as an $\ell$-dimensional vector
$$
\mu = (\mu_1,\dots,\mu_\ell)\in \bN^\ell
$$
consisting with positive integers. 
We define
\begin{multline}
\label{eq:Hg(mu)}
H_g(\mu)= \frac{|\Aut(\mu)|}{r(g,\mu)!}\cdot h_{g,\mu}
\\
=
  \prod_{i=1}^{\ell}\frac{\mu_i ^{\mu_i}}{\mu_i!}\int_{\overline{\mathcal{M}}_{g,\ell}} \frac{\Lambda_g^{\vee}(1)}{\prod_{i=1}^{\ell}\big( 1-\mu_i \psi_i\big)}
=
\sum_{n_1+\cdots+n_\ell \le 3g-3+\ell} \la \tau_{n_1}\cdots \tau_{n_\ell}\Lambda_g^{\vee}(1)\ra\,\,\prod_{i=1}^{\ell}\frac{\mu_i^{\mu_i+n_i}}{\mu_i!}
\end{multline}
as  a function in $\mu$.
It's Laplace transform
\begin{equation}
\label{eq:Hgell-in-w}
\cH_{g,\ell}(w_1,\dots,w_\ell)
=
\sum_{\mu\in\bN^\ell} H_g(\mu)
e^{-\left(
\mu_1(w_1+1)+\cdots+\mu_\ell(w_\ell+1)
\right)}
\end{equation}
is the function we consider in this paper.
We note that the automorphism group 
$\Aut (\mu)$ acts trivially on the function
$e^{-\left(
\mu_1(w_1+1)+\cdots+\mu_\ell(w_\ell+1)
\right)}$, which explains its appearance in (\ref{eq:Hg(mu)}).
Since
the coordinate change
$x = e^{-(w+1)}$
identifies
\begin{equation}
\label{eq:H-equality}
\frac{\partial ^\ell}{\partial x_1\cdots \partial x_\ell}
\cH_{g,\ell}\big(w(x_1),\dots,w(x_\ell)\big)
=H_{g,\ell}(x_1,\dots,x_\ell),
\end{equation}
the Laplace transform (\ref{eq:Hgell-in-w}) is a 
\emph{primitive} of the generating function (\ref{eq:Hgell}).

Before performing the exact calculation 
of the holomorphic function
 $\cH_{g,\ell}(w_1,\dots,w_\ell)$, let us
 make a quick estimate here. 
 From Stirling's formula
$$
e^{-k}\frac{k^{k+n}}{k!}\sim \frac{1}{\sqrt{2\pi}}\;k^{n-\half},
$$
  it is obvious that 
$\cH_{g,\ell}(w_1,\dots,w_\ell)$ is  holomorphic on
$Re(w_i) >0$ for all $i=1,\dots,\ell$.  
Because of the  half-integer powers of $\mu_i$'s, 
 the Laplace transform 
$\cH_{g,\ell}(w_1,\dots,w_\ell)$
 is expected to be
  a meromorphic function on a double-sheeted covering
 of the $w_i$-planes.
Such a double covering is provided by the Lambert curve
$C$ of (\ref{eq:Lambert-xy}). 
So we define
\begin{equation}
\label{eq:t-in-w}
t= 1+
\sum_{k=1} ^\infty \frac{k^{k}}{k!}
e^{-k(w+1)},
\end{equation}
which gives a global coordinate of $C$.
The summation converges for $Re(w)>0$, and
the Lambert curve 
is expressed in terms of $w$ and $t$ 
coordinates as
\begin{equation}
\label{eq:Lambert-tw}
e^{-(w+1)} = 
\left(1-\frac{1}{t}\right) e^{-\left(1-\frac{1}{t}\right)} .
\end{equation}
The $w$-projection $\pi:C\rightarrow \bC$ is 
locally a double-sheeted covering at $t=\infty$. 
The inverse function of (\ref{eq:t-in-w}) is given by
\begin{equation}
\label{eq:w-in-t}
 w =w(t)= -\frac{1}{t} - \log \left(1-\frac{1}{t}\right)
=\sum_{m=2} ^\infty \frac{1}{m}\;\frac{1}{t^m} ,
\end{equation}
which is holomorphic on $Re(t)>1$. 
When considered as a functional equation, 
(\ref{eq:w-in-t}) has exactly two solutions:
$t$ and
\begin{equation}
\label{eq:s}
s(t) =  -t + \frac{2}{3} + 
\frac{4}{135}t^{-2}+\frac{8}{405} t^{-3}+\frac{8}{567}t^{-4}
+\cdots .
\end{equation}
This is the deck-transformation of 
 the projection $\pi:C\rightarrow \bC$ near $t=\infty$ 
 and satisfies the involution equation
 $s\big(s(t)\big)= t$.
  It is analytic on
 $\bC\setminus [0,1]$ and has logarithmic singularities
 at $0$ and $1$. 
Although $w(t)=w\big(s(t)\big)$,  
$s(t)$ is not given by the Laplace transform (\ref{eq:t-in-w}).

Since the Laplace transform
\begin{equation}
\label{eq:xihat}
\hat{\xi}_n(t)=
\sum_{k=1} ^\infty  \frac{k^{k+n}}{k!}
e^{-k(w+1)}
\end{equation}
 also naturally lives  on $C$, it is a meromorphic
 function  in $t$ rather than in $w$. 
Actually it is a \emph{polynomial}
of degree $2n+1$ for $n\ge 0$
because of the recursion formula
\begin{equation}
\label{eq:hxi-recursion}
\hxi_{n+1}(t) = t^2(t-1)\frac{d}{dt}\;\hxi_n(t) \qquad {\text{for all\;\;}}n\ge 0,
\end{equation}
which follows from (\ref{eq:t-in-w}), (\ref{eq:xihat}), 
and (\ref{eq:w-in-t}) that implies
\begin{equation}
\label{eq:dw}
-dw=\frac{dt}{t^2(t-1)}.
\end{equation}
We note that the differential operator of (\ref{eq:hxi-recursion})
is discovered in \cite{GJV2}.
 For future convenience, 
we define
\begin{equation}
\label{eq:hxi-1}
\hxi_{-1}(t) = \frac{t-1}{t} = y,
\end{equation}
which is indeed the $y$ coordinate of the original 
Lambert curve (\ref{eq:Lambert-xy}). 
We now see that the 
Laplace transform  
\begin{multline}
\label{eq:Hhatgell}
\hatH_{g,\ell}(t_1,\dots,t_\ell) = 
\cH_{g,\ell}\big((w(t_1),\dots,w(t_\ell)\big)
=
\sum_{\mu\in\bN^\ell} H_g(\mu)
e^{-\left(
\mu_1(w_1+1)+\cdots+\mu_\ell(w_\ell+1)
\right)}
\\
=
\sum_{n_1+\cdots+n_\ell \le 3g-3+\ell} \la \tau_{n_1}\cdots \tau_{n_\ell}\Lambda_g^{\vee}(1)\ra\,\,\prod_{i=1}^{\ell}
\hxi_{n_i}(t_i)
\end{multline}
is  a symmetric \emph{polynomial} in the $t$-variables
when $2g-2+\ell>0$.

It has been noted in \cite{BEMS, E2007, E2008, 
EO1, EO2} that the \emph{Airy curve}
$w = \half v^2$  
is a universal object of  the topological recursion 
for the case of a genus $0$ spectral curve  
with only one critical point. Analysis of the Airy curve
provides a good control of the topological 
recursion formula for such cases. 
The Airy curve expression  is also valid
around any non-degenerate critical point of a
general spectral curve.  
To switch to the local Airy curve coordinate,
we define
\begin{equation}
 \label{eq:v(t)}
v= v(t) = t^{-1} + \frac{1}{3}\; t^{-2} +\frac{7}{36}\;t^{-3}
+\frac{73}{540}\; t^{-4} + \frac{1331}{12960}\;t^{-5}+\cdots 
 \end{equation}
 as a function in $t$ that solves
\begin{equation}
\label{eq:tvw}
\half v^2 = w = -\frac{1}{t} - \log \left(1-\frac{1}{t}\right)
= -\frac{1}{s(t)} - \log \left(1-\frac{1}{s(t)}\right)
=\sum_{m=2} ^\infty \frac{1}{m}\;\frac{1}{t^m} .
\end{equation}
Note that we are making  a choice of the branch of the
square root of $w$ 
that is consistent with  (\ref{eq:t-in-w}).
The involution (\ref{eq:s}) becomes simply
\begin{equation}
\label{eq:v(s)}
v(t) = -v\big(s(t)\big).
\end{equation}
The new coordinate $v$ plays a key role later when we
reduce
the Laplace transform of the cut-and-join equation
(\ref{eq:main}) to the Bouchard-Mari\~no
topological recursion (\ref{eq:BMalgebraic}).

\section{The cut-and-join equation and its 
Laplace transform}
\label{sect:CAJ}

In the modern times
the cut-and-join equation for Hurwitz numbers was
discovered in 
\cite{GJ, V}, though it seems to be known to 
Hurwitz \cite{H}.  It  has become an effective
tool for studying algebraic geometry of Hurwitz 
numbers and many related subjects \cite{CLL, GJV1, GJV2,
GJV3, KazarianLando, KL, LZZ,
 LLZ, OP1, Zhou}. 
In this section we calculate the Laplace transform of 
the  cut-and-join 
equation 
following \cite{MZ}.

The simplest way of presenting the cut-and-join
equation is to use a different \emph{primitive}
 of the same generating function 
of Hurwitz numbers (\ref{eq:Hgell}). Let
\begin{equation}
\label{eq:anotherHurwitz}
\mathbf{H}(s,\mathbf{p})=\sum_{g\ge 0}\sum_{\ell\ge 1} \mathbf{H}_{g,\ell}
(s,\mathbf{p});
\qquad
\mathbf{H}_{g,\ell}
(s,\mathbf{p}) =\sum_{\mu: \ell(\mu)=\ell} h_{g,\mu}\mathbf{p}_{\mu} \frac{s^r}{r!},
\end{equation}
where $\mathbf{p}_{\mu} = p_{\mu_1}p_{\mu_2}\cdots p_{\mu_{\ell}}$,
and $r=2g-2+\ell+|\mu|$ is the number of simple ramification 
points on $\bP^1$. The summation is over all partitions of length $\ell$. 
Here $p_k$ is the power-sum symmetric function
\begin{equation}
\label{eq:pk}
p_k = \sum_{i\ge 1} x_i ^k ,
\end{equation}
which is related to the 
monomial symmetric functions by
\begin{equation*}
\frac{\partial ^\ell}{\partial x_1\cdots \partial x_\ell} 
\mathbf{p}_\mu
=\sum_{\sigma\in S_\ell} 
\prod_{i=1} ^\ell 
\mu_i \,x_{\sigma(i)} ^{\mu_i-1} .
\end{equation*}
Therefore, we have
\begin{equation*}
\frac{\partial ^\ell}{\partial x_1\cdots \partial x_\ell} 
\mathbf{H}_{g,\ell}
(1,\mathbf{p})=
H_{g,\ell}(x_1,\dots,x_\ell)
=
\sum_{\mu : \ell(\mu) = \ell}
\frac{\mu_1\mu_2\cdots\mu_\ell}{(2g - 2 +\ell +|\mu|)!}\; h_{g,\mu}
\sum_{\sigma\in S_\ell}\prod_{i=1} ^\ell 
x_{\sigma(i)}  ^{\mu_i -1},
\end{equation*}
which is the generating
function of (\ref{eq:Hgell}). 
Because of the identification  (\ref{eq:H-equality}),
 the primitives 
 $\mathbf{H}_{g,\ell}
(1,\mathbf{p}) $ and $\widehat{\cH}_{g,\ell}
(t_1,\dots,t_\ell)$ 
of (\ref{eq:Hhatgell})
are essentially the same function, 
different only by a constant.

\begin{rem}
Although we do not use the fact here, we note that 
$\mathbf{H}(s,\mathbf{p})$ is a one-parameter 
family of $\tau$-functions of the
KP equations 
with $\frac{1}{k}p_k$ as the KP time variables
\cite{KazarianLando, O2000}. The parameter $s$ 
is the deformation parameter.
 \end{rem}

Let $z\in \bP^1$ be a point at which the covering $f:X\rightarrow \bP^1$
is simply ramified. Locally we can name sheets, so we assume
sheets $a$ and $b$ are ramified over $z$. When we 
merge $z$ to $\infty$, one of the 
two things happen:

\begin{enumerate}
\item The \emph{cut case}. If both sheets are ramified at
  the same point
$x_i$
of the inverse image $f^{-1}(\infty) = \{x_1,\dots, x_\ell\}$, then 
the resulting ramification after merging $z$ to $\infty$ has a profile 
$$
 (\mu_1,\dots,\widehat{\mu_i},\dots,  \mu_{\ell},\alpha, \mu_i-\alpha) = \big(\mu(\hat{i}) , \a, \mu_i-\a\big)
$$
for $1\le \alpha< \mu_i$. 
\item   Otherwise we are in the
\emph{join case}. If sheets $a$ and $b$ are ramified at two distinct
points, say $x_i$ and $x_j$ above $\infty$, then the result of merging 
creates a new profile
$$
  (\mu_1,\dots,\widehat{\mu_i},\dots,\widehat{\mu_j},\dots,\mu_\ell,
\mu_i+\mu_j) 
=\big(\mu(\hat{i},\hat{j}), \mu_i+\mu_j\big).
$$
\end{enumerate}
Here the $\widehat{\;\;}$ sign means removing the entry. 
The above consideration tells us what happens to the generating
function of the Hurwitz numbers when we differentiate it by 
$s$, because it decreases the degree in $s$, or the number
of simple ramification points, by $1$. 
Since the cut case may cause a disconnected covering, let us use
the generating function of Hurwitz numbers allowing 
disconnected curves to cover $\bP^1$. Then the cut-and-join equation
takes the following simple form:
\begin{equation*}
\left[\frac{\partial}{\partial s}-\frac{1}{2}\sum_{\alpha,\beta\ge 1}
\left((\alpha+\beta)p_\alpha p_\beta \frac{\partial}{\partial p_{\alpha+\beta}} + 
\alpha \beta p_{\alpha+\beta}\frac{\partial^2}{\partial p_\alpha\partial p_\beta}\right)\right]
 e^{\mathbf{H}(s,\mathbf{p})} = 0 .
\end{equation*}
It immediately implies
\begin{equation}
\label{eq:cajconnected}
\frac{\partial \mathbf{H}}{\partial s}
=\frac{1}{2}\sum_{\alpha, \beta\ge 1}
\left((\alpha+\beta)p_\alpha p_\beta \frac{\partial \mathbf{H}}{\partial p_{\alpha+\beta}} +
\alpha\beta p_{\alpha+\beta}\frac{\partial^2 \mathbf{H}}{\partial p_\alpha\partial p_\beta}
+ 
\alpha\beta p_{\alpha+\beta}\frac{\partial \mathbf{H}}{\partial p_\alpha}\cdot 
\frac{\partial \mathbf{H}}{\partial p_\beta}\right) ,
\end{equation}
which is the cut-and-join equation for the generating function
$\mathbf{H}(s,\mathbf{p})$ of the number of 
\emph{connected} Hurwitz coverings.

Let us now apply the ELSV formula (\ref{eq:ELSV}) to 
(\ref{eq:anotherHurwitz}). We obtain
\begin{multline}
\label{eq:Hgell-bold}
\mathbf{H}_{g,\ell}(s,\mathbf{p}) = \frac{1}{\ell!}
\sum_{n_L\in\bN^\ell} \la \tau_{n_L}\Lambda^\vee _g(1)\ra_{g,\ell}
s^{2g-2+\ell} \prod_{i=1} ^\ell
\sum_{\mu_i=1} ^\infty \frac{\mu_i ^{\mu_i + n_i}}{\mu_i !}
s^{\mu_i} p_{\mu_i}
\\
= \frac{1}{\ell!}\sum_{(\mu_1,\dots,\mu_\ell)\in \bN^\ell} 
H_g(\mu) \mathbf{p}_\mu s^r  
=\sum_{\mu:\ell(\mu)=\ell} \frac{1}{|\Aut(\mu)|}\,
H_g(\mu) \mathbf{p}_\mu s^r  ,
\end{multline}
where $H_g(\mu)$ is introduced in (\ref{eq:Hg(mu)}). 
Now for every choice of $r\ge 1$ and 
a partition $\mu$, the coefficient of
$\mathbf{p}_\mu s^{r-1}$ of the cut-and-join equation
(\ref{eq:cajconnected}) gives

\begin{thm}[\cite{MZ}]
\label{thm:caj}
The functions $H_g(\mu)$ of {\rm{(\ref{eq:Hg(mu)})}}
satisfy a recursion equation
\begin{multline}
\label{eq:cajcoefficient}
r(g,\mu){H}_g  (\mu)
=
\sum_{i< j}
(\mu_i+\mu_j)
{H}_g \big(\mu(\hat{i},\hat{j}),\mu_i+\mu_j\big)
 \\
 + \frac{1}{2} \sum_{i=1} ^\ell
 \sum_{\alpha+\beta = \mu_i}\alpha\beta 
\left({H}_{g-1} \big(\mu(\hat{i}),\a,\b\big)
+
\sum_{\substack{g_1+g_2 = g\\
\nu_1\sqcup \nu_2 = \mu(\hat{i})}}
{H}_{g_1} (\nu_1,\a)
{H}_{g_2} (\nu_2,\b)
\right).
\end{multline}
\end{thm}

\begin{rem}
Note that
\begin{align*}
\ell\big(\mu(\hat{i},\hat{j})\big) &= \ell-2\\
\ell(\nu_1) + \ell(\nu_2) &= \ell\big(\mu(\hat{i})\big) =\ell -1 .
\end{align*}
Thus the complexity $2g-2+\ell$ is one less for the coverings
appearing in the RHS of (\ref{eq:cajcoefficient}), which is the effect
of $\partial/\partial s$ applied to $\mathbf{H}(s,\mathbf{p})$,
 \emph{except} for the unstable geometry
corresponding to $g_i=0$ and $|\nu_i|=0$
in the join terms. If we move the $(0,1)$-terms 
to the LHS, then 
 the cut-and-join equation (\ref{eq:cajcoefficient}) becomes 
a topological recursion formula. 
\end{rem}

Let us first calculate the
Laplace transform of the cut-and-join equation
for  the $\ell=1$ case to see what is involved.
We then move on to the more general case later,
following \cite{MZ}.

 \begin{prop}
 \label{prop:ell=1LT}
 The Laplace transform of the cut-and-join equation
 for the $\ell=1$ case 
 gives the following equation:
 \begin{multline}
\label{eq:ell=1LT}
\sum_{n\le 3g-2}\la \tau_n \Lam_g ^\vee (1)\ra_{g,1}
\left[
(2g-1)\hxi_n(t)+\hxi_{n+1}(t)\big(1 -\hxi_{-1}(t)\big)
\right]
\\
=
\half
\sum_{a+b\le 3g-4}
\left[
\la \tau_a\tau_b \Lam_{g-1} ^\vee (1)\ra_{g-1,2}
+\sum_{g_1+g_2=g} ^{\rm{stable}}
\la \tau_a\Lam_{g_1} ^\vee(1) \ra_{g_1,1}
\la \tau_b\Lam_{g_2} ^\vee(1) \ra_{g_2,1}
\right]
\hxi_{a+1}(t)\hxi_{b+1}(t).
\end{multline}
\end{prop}

\begin{proof}
The cut-and-join
equation  for $\ell=1$ is a simple equation
\begin{equation}
\label{eq:caj-ell=1}
(2g-1+\mu) H_g(\mu) 
= \half \sum_{\a+\b=\mu} \a\b
\left(
{H}_{g-1}(\a,\b)
 + \sum_{g_1+g_2=g}H_{g_1}(\a)
H_{g_2}(\b)
\right).
\end{equation}
The Laplace transform of the LHS of (\ref{eq:caj-ell=1}) is
\begin{equation*}
\sum_{n\le 3g-2}\la \tau_n \Lam_g ^\vee (1)\ra_{g,1}
\left[
(2g-1)\hxi_n(t)+\hxi_{n+1}(t)
\right].
\end{equation*}
When summing over $\mu$ to compute the Laplace transform
of the RHS, we switch to sum over $\a$ and $\b$ independently.
The  factor $\half$ cancels the double count
on the diagonal. Thus
the Laplace transform of the \emph{stable}
geometries of the RHS is
\begin{equation*}
\half
\sum_{a+b\le 3g-4}
\left[
\la \tau_a\tau_b \Lam_{g-1} ^\vee (1)\ra_{g-1,2}
+\sum_{g_1+g_2=g} ^{\rm{stable}}
\la \tau_a\Lam_{g_1} ^\vee(1) \ra_{g_1,1}
\la \tau_b\Lam_{g_2} ^\vee(1) \ra_{g_2,1}
\right]
\hxi_{a+1}(t)\hxi_{b+1}(t).
\end{equation*}
The \emph{unstable} terms contained
in the second summand
of the RHS of (\ref{eq:caj-ell=1}) are the
 $g=0$ terms  $H_0(\a)H_g(\b)+H_g(\a)H_0(\b)$. 
We calculate the Laplace transform of these unstable terms 
using (\ref{eq:01Hodge}). Since
$$
H_0(\a) = \frac{\a^{\a-2}}{\a!},
$$
the result is
\begin{equation*}
\sum_{a}
\la \tau_a\Lam_{g} ^\vee(1) \ra_{g,1}\hxi_{-1}(t)\hxi_{a+1}(t).
\end{equation*}
This completes the proof.
\end{proof}

\begin{rem}
We note that (\ref{eq:ell=1LT}) is a polynomial
equation of degree $2n+2$. Since $\hxi_{-1}(t) = 1-\frac{1}{t}$,
the leading term of $\hxi_{n+1}(t)$ is canceled in the formula.
\end{rem}

To calculate the Laplace transform of the general 
case
(\ref{eq:cajcoefficient}), we need
to deal with both of the  unstable geometries $(g,\ell) = (0,1)$ and 
$(0,2)$. 
These are the exceptions for the general formula
(\ref{eq:Hhatgell}).
Recall the $(0,1)$ case (\ref{eq:01Hodge}). 
The formula
\begin{equation}
\label{eq:H01LT}
\hatH_{0,1}(t) = \sum_{k=1} ^\infty \frac{k^{k-2}}{k!}
e^{-k(w+1)} = -\frac{1}{2 \,t^2}+c = \hxi_{-2}(t),
\end{equation}
where the constant $c$ is given  by
$$
c=\sum_{k=1} ^\infty \frac{k^{k-2}}{k!}e^{-k},
$$
is used in (\ref{eq:ell=1LT}).
The  $(g,\ell) =(0,2)$ terms require a more careful computation.
We shall see that these are the terms that exactly
correspond to the terms involving the Cauchy differentiation kernel
in the Bouchard-Mari\~no recursion.

\begin{prop}
\label{prop:H02}
We have the following  Laplace transformation formula:
\begin{multline}
\label{eq:H02LT}
\widehat{\cH}_{0,2}(t_1,t_2) =
\sum_{\mu_1,\mu_2\ge 1}
\frac{1}{\mu_1+\mu_2}\cdot\frac{\mu_1 ^{\mu_1}}{\mu_1 !}\cdot\frac{\mu_2 ^{\mu_2}}{\mu_2 !}
e^{-\mu_1 (w_1+1)}e^{-\mu_2 (w_2+1)}\\
=\log\left(
\frac{\hxi_{-1}(t_1)-\hxi_{-1}(t_2)}
{x_1-x_2}
\right)
-\hxi_{-1}(t_1)-\hxi_{-1}(t_2).
\end{multline}
\end{prop}

\begin{proof}
Since $x = e^{-(w+1)}$, (\ref{eq:H02LT}) is equivalent 
to 
\begin{equation}
\label{eq:H02-inx}
\sum_{\substack{\mu_1,\mu_2\ge 0\\(\mu_1,\mu_2)\ne (0,0)}}
\frac{1}{\mu_1+\mu_2}\cdot\frac{\mu_1 ^{\mu_1}}{\mu_1 !}\cdot\frac{\mu_2 ^{\mu_2}}{\mu_2 !}
x_1 ^{\mu_1}x_2 ^{\mu_2}
=\log\left(
\sum_{k=1} ^\infty \frac{k^{k-1}}{k!}\cdot
\frac{x_1 ^k - x_2 ^k}{x_1-x_2}
\right),
\end{equation}
where $|x_1|< e^{-1}, |x_2|<e^{-1}$, and $0<|x_1-x_2|<e^{-1}$
so that the formula is an equation of holomorphic functions 
in $x_1$ and $x_2$.
Define
\begin{equation*}
\phi(x_1,x_2)\overset{{\rm{def}}}{=}\sum_{\substack{\mu_1,\mu_2\ge 0\\(\mu_1,\mu_2)\ne (0,0)}}
\frac{1}{\mu_1+\mu_2}\cdot\frac{\mu_1 ^{\mu_1}}{\mu_1 !}\cdot\frac{\mu_2 ^{\mu_2}}{\mu_2 !}
x_1 ^{\mu_1}x_2 ^{\mu_2}
-\log\left(
\sum_{k=1} ^\infty \frac{k^{k-1}}{k!}\cdot
\frac{x_1 ^k - x_2 ^k}{x_1-x_2}
\right).
\end{equation*}
Then
\begin{multline*}
\phi(x,0)=\sum_{\mu_1\ge 1}
\frac{\mu_1 ^{\mu_1 -1}}{\mu_1 !}
x ^{\mu_1}
-\log\left(
\sum_{k=1} ^\infty \frac{k^{k-1}}{k!}\cdot
x ^{k-1}
\right)
\\
=
\hxi_{-1}(t) - \log\left(\frac{\hxi_{-1}(t)}{x}\right)
=
1-\frac{1}{t} -\log\left(1-\frac{1}{t}\right) +\log x
\\
=
1-\frac{1}{t} -\log\left(1-\frac{1}{t}\right) - w-1 = 0
\end{multline*}
due to (\ref{eq:w-in-t}). Here $t$ is restricted on the 
domain $Re(t)>1$. 
Since
\begin{multline*}
x_1 \frac{\partial}{\partial x_1}
\log\left(
\sum_{k=1} ^\infty \frac{k^{k-1}}{k!}\cdot
\frac{x_1 ^k - x_2 ^k}{x_1-x_2}
\right)
\\
=
t_1 ^2 (t_1-1)\frac{\partial}{\partial t_1}
\log\left(\hxi_{-1}(t_1)-\hxi_{-1}(t_2)\right)
-x_1 \frac{\partial}{\partial x_1}\log(x_1-x_2)
\\
=
t_1 ^2 (t_1-1)\frac{\partial}{\partial t_1}
\log\left(-\frac{1}{t_1}+\frac{1}{t_2}\right)
- \frac{x_1}{x_1-x_2}
\\
=
\frac{t_1t_2(t_1-1)}{t_1-t_2}-\frac{x_1}{x_1-x_2},
\end{multline*}
we have 
\begin{multline*}
\left(x_1 \frac{\partial}{\partial x_1}+
x_2 \frac{\partial}{\partial x_2}
\right)
\log\left(
\sum_{k=1} ^\infty \frac{k^{k-1}}{k!}\cdot
\frac{x_1 ^k - x_2 ^k}{x_1-x_2}
\right)
\\
=
\frac{t_1t_2(t_1-1)-t_1t_2(t_2-1)}{t_1-t_2}
-\frac{x_1-x_2}{x_1-x_2}
\\
=
t_1t_2 - 1 = \hxi_0(t_1)\hxi_0(t_2)+\hxi_0(t_1)+\hxi_0(t_2).
\end{multline*}
On the other hand, we also have
\begin{multline*}
\left(x_1 \frac{\partial}{\partial x_1}+
x_2 \frac{\partial}{\partial x_2}
\right)
\sum_{\substack{\mu_1,\mu_2\ge 0\\(\mu_1,\mu_2)\ne (0,0)}}
\frac{1}{\mu_1+\mu_2}\cdot\frac{\mu_1 ^{\mu_1}}{\mu_1 !}\cdot\frac{\mu_2 ^{\mu_2}}{\mu_2 !}
x_1 ^{\mu_1}x_2 ^{\mu_2}
\\
=
\hxi_0(t_1)\hxi_0(t_2)+\hxi_0(t_1)+\hxi_0(t_2).
\end{multline*}
Therefore, 
\begin{equation}
\label{eq:Eulerequation}
\left(x_1 \frac{\partial}{\partial x_1}+
x_2 \frac{\partial}{\partial x_2}
\right)\phi(x_1,x_2) = 0.
\end{equation}
Note that $\phi(x_1,x_2)$ is a holomorphic function 
in $x_1$ and $x_2$. Therefore, it has a series
expansion in homogeneous polynomials around $(0,0)$. 
Since a homogeneous polynomial in $x_1$ and $x_2$ of
degree $n$ is an eigenvector of the differential operator 
$x_1 \frac{\partial}{\partial x_1}+
x_2 \frac{\partial}{\partial x_2}$ belonging to the 
eigenvalue $n$, the only holomorphic
solution to the Euler differential
equation (\ref{eq:Eulerequation}) is a constant. But 
since $\phi(x_1,0)=0$, we conclude that 
$\phi(x_1,x_2)=0$. This completes the proof of 
(\ref{eq:H02-inx}), and hence the proposition.
\end{proof}

The following polynomial recursion  formula
was established in \cite{MZ}. Since
each of the polynomials 
$\hatH_{g,\ell}(t_L)$'s in (\ref{eq:cajLT})
satisfies the stability condition $2g-2+\ell>0$, 
it is equivalent to (\ref{eq:main}) after 
expanding the generating functions using 
(\ref{eq:Hhatgell}).

 \begin{thm}[\cite{MZ}]
 \label{thm:CAJLT}
The Laplace transform of the 
cut-and-join equation {\rm{(\ref{eq:cajcoefficient})}}
produces the following polynomial equation on the Lambert curve:\begin{multline}
\label{eq:cajLT}
\left(
2g-2+\ell +\sum_{i=1} ^\ell
\big(1-\hxi_{-1}(t_i)\big)
t_i ^2 (t_i-1)\frac{\partial }{\partial t_i}
\right)
\hatH_{g,\ell}(t_1,\dots,t_\ell)
\\
=
\sum_{i< j}t_it_j
 \frac{  t_i ^2(t_i-1)^2\frac{\partial}{\partial t_i}
    \hatH_{g,\ell-1}\left(t_1,\dots,\widehat{t_j},\dots,t_\ell\right)
    -
      t_j ^2(t_j-1)^2\frac{\partial}{\partial t_j}
    \hatH_{g,\ell-1}\left(t_1,\dots,\widehat{t_i},\dots,t_\ell\right)}{t_i-t_j}
   \\
- \sum_{i\ne j} t_i ^3(t_i-1)\frac{\partial}{\partial t_i}
\hatH_{g,\ell-1}\left(t_1,\dots,\widehat{t_j},\dots,t_\ell\right)
\\
+
\half
\sum_{i=1} ^\ell
\left[
u_1 ^2 (u_1-1) u_2 ^2(u_2-1)
\frac{\partial ^2}{\partial u_1\partial u_2}
\hatH_{g-1,\ell+1}\left(u_1,u_2,t_{L\setminus\{i\}}\right)
\right]_{u_1=u_2=t_i}
\\
+
\half
\sum_{i=1} ^\ell
\sum_{\substack{g_1+g_2 = g\\
J\sqcup K= L\setminus\{i\}}} ^{\rm{stable}}
t_i ^2 (t_i-1) 
\frac{\partial }{\partial t_i}
\hatH_{g_1,|J|+1}(t_i,t_J)\cdot 
t_i ^2 (t_i-1) 
\frac{\partial }{\partial t_i}
\hatH_{g_2,|K|+1}(t_i,t_K) .
\end{multline}
In the last sum each term is restricted to 
satisfy the stability conditions
$2g_1-1+|J|>0$ and $2g_2-1+|K|>0$.
\end{thm}

\begin{rem}
The polynomial equation (\ref{eq:cajLT}) is 
equivalent to the original cut-and-join
equation (\ref{eq:cajcoefficient}).
Note that the topological recursion structure of 
(\ref{eq:cajLT}) is exactly the same as (\ref{eq:BMalgebraic}).
Although (\ref{eq:cajLT}) contains more terms,
all functions involved are polynomials that are easy to 
calculate from (\ref{eq:hxi-recursion}), whereas 
 (\ref{eq:BMalgebraic}) requires computation
of the involution $s(t)$ of (\ref{eq:s}) and infinite
series expansions.
\end{rem}

\begin{rem}
\label{rem:WK}
It is an easy task to deduce the
Witten-Kontsevich theorem, i.e., the Virasoro constraint 
condition for the $\psi$-class intersection numbers
\cite{W1991, K1992}, from (\ref{eq:cajLT}).
Let us use the normalized
notation $\sigma_n = (2n+1)!!\tau_n$ for the $\psi$-class
intersections. 
Then the formula  according to Dijkgraaf, Verlinde and Verlinde
\cite{DVV} is 
\begin{multline}
\label{eq:DVV}
\la \sigma_n\sigma_{n_L}\ra _{g,\ell+1}=
\frac{1}{2} \, \sum_{a+b=n-2}
\la \sigma_a\sigma_b\sigma_{n_L}\ra _{g-1,\ell+2}
+
\sum_{i\in L} (2n_i +1) \la \sigma_{n+n_i-1}
 \sigma_{n_{L\setminus \{i\}}}
\ra _{g,\ell}\\
+\frac{1}{2}
\sum_{\substack{g_1+g_2=g\\
I\sqcup J= L}} ^{
\text{stable}}
\sum_{a+b=n-2}
\la \sigma_a\sigma_{n_I}\ra _{g_1,|I|+1}\cdot
\la \sigma_b\sigma_{n_J}\ra _{g_2,|J|+1}.
\end{multline}
 Eqn.(\ref{eq:DVV}) is exactly the relation of the homogeneous  top
 degree terms of (\ref{eq:cajLT}),
 after canceling the
 highest degree terms coming from 
 $\hxi_{n_i+1}(t_i)$ in the LHS \cite{MZ}.
 This derivation is in the same spirit as those found in 
 \cite{CLL, KL, OP1}, though the argument is much clearer
 due to the polynomial nature of our equation.
 \end{rem}

\section{The Bouchard-Mari\~no recursion formula
for Hurwitz numbers}
\label{sect:BM}

In this section we present the precise statement of
the Bouchard-Mari\~no conjecture on Hurwitz numbers.

Recall the function we introduced in (\ref{eq:t-in-w}):
\begin{equation}
\label{eq:t}
t =  t(x) = 1+ \sum_{k=1} ^\infty \frac{k^k}{k!} x^k
\end{equation}
This is closely related to 
the \emph{Lambert  $W$-function}
\begin{equation}
\label{eq:LambertW}
W(x) = -\sum_{k=1}^{\infty}
\frac{k^{k-1}}{k!}(-x)^{k}.
\end{equation}
By abuse of terminology, we also call the function $t(x)$ 
of (\ref{eq:t}) the Lambert function. 
The power series (\ref{eq:t}) has the radius of convergence
$1/e$, and its  inverse function is given by
\begin{equation}
\label{eq:x}
x = x(t) = \frac{1}{e} \left(1-\frac{1}{t}\right)e^{\frac{1}{t}}.
\end{equation} 
Motivated by the Lambert $W$-function, 
 a plane analytic curve
\begin{equation}
\label{eq:curve}
C = \{(x,t)\;|\; x = x(t)\} \subset \bC ^* \times \bC ^*
\end{equation}
is introduced in \cite{BM}, which is exactly the Lambert
curve (\ref{eq:Lambert-xy}).
We denote by $\pi:C\rightarrow \mathbb{C}$ the $x$-projection. 
 Bouchard-Mari\~no \cite{BM} then defines
a tower of polynomial differentials on the Lambert curve
$C$ by
\begin{equation}
\label{eq:xin-recursion}
\xi_n(t) = \frac{d}{dt} \bigg[t^2(t-1)\;\xi_{n-1}(t)\bigg]
\end{equation}
with the initial condition
\begin{equation}
\label{eq:xi0}
\xi_0(t) = dt .
\end{equation}
It is obvious from (\ref{eq:xin-recursion}) and
(\ref{eq:xi0}) that for $n\ge 0$, $\xi_n(t)$ is
a polynomial $1$-form of degree $2n$ with a general expression
\begin{equation*}
\xi_n(t)
= t^n\left[
(2n+1)!!\, t^n - \left(\frac{(2n+3)!!}{3}-(2n+1)!!\right)
t^{n-1} + \cdots + (-1)^n (n+1)!\right] dt.
\end{equation*}
 All the coefficients of $\xi_n(t)$ have
a combinatorial meaning called the \emph{second order reciprocal
Stirling numbers}. As we will note below, the leading
coefficient is responsible for the Witten-Kontsevich theorem
on the cotangent class intersections, and the lowest coefficient
is related to the $\lambda_g$-formula \cite{MZ}. For a convenience,
we also use $\xi_{-1}(t) = t^{-2}\, dt$
and $\xi_{-2}(t) = t^{-3}\, dt$.

\begin{rem}
The polynomial  $\hxi_n(t)$ of (\ref{eq:xihat}) is a primitive
of $\xi_n(t)$:
\begin{equation}
\label{eq:dhxi=xi}
d\hxi_n(t) = \xi_n(t).
\end{equation}
\end{rem}

\begin{Def}
\label{Def:Hdiff}
Let us call   the symmetric polynomial
differential form  
\begin{equation*}
d^{\tensor \ell}\,\hatH_{g,\ell}(t_1,\dots,t_\ell) =
\sum_{ n_1+\cdots +n_\ell \leq 3g-3+\ell} 
\la \tau_{n_1}\cdots \tau_{n_\ell}\Lambda_g^{\vee}(1)\ra\,\bigotimes_{i=1}^{\ell} \xi_{n_i}(t_i)
\end{equation*}
on $C^\ell$ the \emph{Hurwitz differential} of type
$(g,\ell)$.
\end{Def}

\begin{rem}
Our $\xi_n(t)$ is exactly the same as the $\zeta_n(y)$-differential of \cite{BM}.
However, this  mere coordinate change happens to be
essential. Indeed,
 the
fact that our expression is a \emph{polynomial} in $t$-variables
allows us to calculate the residues in the
Bouchard-Mari\~no formula in  Section~\ref{sect:residue}. 
\end{rem}

\begin{rem}
The degree of $d^{\tensor \ell}\,\hatH_{g,\ell}(t_1,\dots,t_\ell)$  is 
$2(3g-3+\ell)$, and  the homogeneous top degree terms
give a generating function of the $\psi$-class intersection
numbers
$$
\sum_{ n_1+\cdots +n_\ell = 3g-3+\ell} 
\la \tau_{n_1}\cdots \tau_{n_\ell}\ra\,
\prod_{i=1} ^\ell
(2n_i+1)!!\; t_i ^{2n_i} 
\bigotimes_{i=1}^{\ell}dt_i .
$$
The homogeneous lowest degree terms of 
$d^{\tensor \ell}\,\hatH_{g,\ell}(t_1,\dots,t_\ell)$ are
$$
(-1)^{3g-3+\ell}
\sum_{ n_1+\cdots +n_\ell = 2g-3+\ell} 
\la \tau_{n_1}\cdots \tau_{n_\ell}\lam_g\ra\,
\prod_{i=1} ^\ell
(n_i+1)!\; t_i ^{n_i} 
\bigotimes_{i=1}^{\ell}dt_i .
$$
The combinatorial coefficients
of the $\lam_g$-formula \cite{FP1, FP2}
can be directly deduced from the topological recursion
formula
(\ref{eq:main})
\cite{MZ}, explaining the mechanism found in \cite{GJV3}.
\end{rem}

\begin{rem}
The unstable Hurwitz differentials follow from  
 (\ref{eq:01Hodge}) and (\ref{eq:H02LT}). They are
\begin{align}
\label{eq:H01}
& d \,\hatH_{0,1}(t)= \frac{1}{t^3}\;dt = \xi_{-2}(t);
\\
\label{eq:H02}
& d^{\tensor 2}\,\hatH_{0,2}(t_1,t_2)= \frac{dt_1\tensor dt_2}{(t_1-t_2)^2}
-\pi^*\frac{dx_1\tensor dx_2}{(x_1-x_2)^2}.
\end{align}
\end{rem}

\begin{rem}
\label{rem:stable}
The simplest  stable Hurwitz differentials are given by
\begin{equation}
\label{eq:H03H11}
\begin{aligned}
&d^{\tensor 3}\,\hatH_{0,3}(t_1,t_2,t_3)=
dt_1\tensor dt_2\tensor dt_3 ;
\\
&d \,\hatH_{1,1}(t)
 = \frac{1}{24}\big( -\xi_0(t) + \xi_1(t)
\big) = \frac{1}{24}(t-1)(3t+1)\,dt.
\end{aligned}
\end{equation}
\end{rem}

The amazing insight of Bouchard and Mari\~no \cite{BM} is that the
Hurwitz differentials of Definition~\ref{Def:Hdiff} should satisfy the
 topological recursion relation of 
 Eynard and Orantin \cite{EO1} based on the 
 analytic curve $C$ of (\ref{eq:curve}) as the spectral curve.  
 Since the topological recursion utilizes the critical 
 behavior of the $x$-projection 
$\pi:C\rightarrow \bC^*$, let us examine the local structure of
$C$ around its
critical points. 
Let $z=-\frac{1}{t}$ be a coordinate of $C$ 
centered at $t=\infty$.
The Lambert curve is then given by
$$
x = \frac{1}{e} (1+z)e^{-z}.
$$
We see that the $x$-projection 
$\pi:C\rightarrow \bC^*$ has a unique critical point $q_0$
at $z=0$. 
Locally around $q_0$
the curve $C$ is a double
cover of $\bC$ branched at $q_0$. For a point
$q\in C$ near $q_0$, let us denote by $\bar{q}$  the 
Galois conjugate
point on $C$ that has the same $x$-coordinate.
 Let 
$S(z)$   be the local deck-transformation of
the covering $\pi:C\rightarrow \bC^*$.  Its defining equation 
\begin{equation}
\label{eq:sz}
S(z)  - \log \big(1+S(z)\big) = z- \log (1+z) = \sum_{m=2} ^\infty
\frac{(-1)^m}{m} z^m
\end{equation}
has a unique analytic solution other than $z$ itself,
which has a branch cut along $(-\infty, -1]$.
We note that  $S$  is an involution $S\big(S(z)\big) = z$, 
and has a Taylor
expansion
\begin{equation*}
S(z) = -z + \frac{2}{3} z^2 -\frac{4}{9} z^3 +
\frac{44}{135} z^4 -\frac{104}{405} z^5
+\frac{40}{189} z^6 -\frac{7648}{42525} z^7 
+ \frac{2848}{18225}z^8
+ O(z^{9})
\end{equation*}
for $|z|<1$.
In terms of the $t$-coordinate, the involution 
corresponds to $s(t)$ of (\ref{eq:s}):
\begin{equation*}
\begin{cases}
t(q) =- \frac{1}{z} = t \\
t(\bar{q}) = -\frac{1}{S(z)} = s(t) 
\end{cases} .
\end{equation*}
The equation (\ref{eq:sz}) defining $S(z)$ translates
 into a relation
  \begin{equation}
 \label{eq:stdifferentialrelation}
 \frac{dt}{t^2(t-1)} = \frac{ds(t)}{ s(t)^2 \big(s(t)-1\big)} =
 -dw=-vdv=
 \pi^*\left( \frac{dx}{x}\right) .
 \end{equation}
Using the global coordinate $t$ of the Lambert curve $C$, 
the Cauchy differentiation kernel (the one called the
 \emph{Bergman kernel}
 in \cite{EO1,BM})  is defined by
\begin{equation}
\label{eq:bergman}
B(t_1,t_2) =  \frac{dt_1\tensor dt_2}{(t_1-t_2)^2} 
=d_{t_1} d_{t_2}\log(t_1-t_2).
\end{equation}
We have  already encountered it in (\ref{eq:H02})
in the expression of $\cH_{0,2}(t_1,t_2)$.
Following \cite{EO1}, define a $1$-form on $C$ by
\begin{multline*}
dE(q,\bar{q}, t_2) = \frac{1}{2} \int_q  ^{\bar{q} }
B(\,\cdot\, , t_2)
= \frac{1}{2} \left(
\frac{1}{t_1 - t_2} - \frac{1}{s(t_1) -t_2}
\right)  dt_2 \\
=\half\bigg( \hat{\xi}_{-1}\big(s(t_1)-t_2\big) 
-\hat{\xi}_{-1}(t_1-t_2)\bigg) dt_2,
\end{multline*}
where the integral is taken with respect to the first 
variable of $B(t_1,t_2)$ along any path from $q$ to 
$\bar{q}$. 
The natural holomorphic symplectic form 
on $\bC^*\times \bC^*$ is given by
\begin{equation*}
\Omega = d\log y \wedge d\log x = d\log\left(1-\frac{1}{ t}
\right) \wedge d\log x .
\end{equation*}
Again following \cite{BM, EO1},
let us introduce another $1$-form on the curve $C$ by 
\begin{equation*}
\omega(q,\bar{q}) =  \int_{q} ^{\bar{q}}
\Omega(\;\cdot\;, x(q))
=\left(\frac{1}{t}-\frac{1}{s(t)}\right)
 \frac{dt}{t^2(t-1)}
= \bigg(
\hat{\xi}_{-1}\big(s(t)\big)-\hat{\xi}_{-1}(t)\bigg)
\frac{dt}{t^2(t-1)} .
\end{equation*}
The kernel operator is defined as the \emph{quotient}
\begin{equation*}
K(t_1,t_2)=
\frac{dE(q,\bar{q}, t_2)}{\omega(q,\bar{q}) }
= \frac{1}{2} \cdot \frac{t_1}{t_1-t_2}\cdot\frac{s(t_1)}{s(t_1)-t_2}
\cdot\frac{t_1 ^2(t_1 -1)}{dt_1}\tensor dt_2, 
\end{equation*}
which is a linear algebraic
operator acting on symmetric differential forms on $C^\ell$
 by
replacing   $dt_1$ with $dt_2$. 
We note that
\begin{equation}
\label{eq:kernelrelation}
K(t_1,t_2) = K\big(s(t_1),t_2\big)\;,
\end{equation}
which follows from (\ref{eq:stdifferentialrelation}). 
In the $z$-coordinate, the kernel has the expression 
\begin{multline}
\label{eq:kernel-in-z}
K= -\frac{1}{2} \cdot \frac{1+z}{z}\cdot
\frac{1}{\big(1+zt\big)\big(1+S(z)t\big)}\cdot 
dt\tensor 
\frac{1}{dz}
\\
= -\frac{1}{2} \cdot \frac{1+z}{z}
\left(
\sum_{m=0} ^\infty
(-1)^m \cdot \frac{z^{m+1} - S(z)^{m+1}}{z-S(z)}
\cdot t^m
\right)
 dt\tensor 
\frac{1}{dz}\\
=-\frac{1}{2} \left(
\frac{1}{z} + 1 +  \frac{1}{3}(3t-2)tz
+ \frac{1}{9}(3t-2)tz^2 \right. \\
 +
\left. \frac{1}{135} (135t^3-180t^2+30t+16)tz^3
+\cdots
\right)  dt\tensor 
\frac{1}{dz}\;.
\end{multline}

\begin{Def}
The \emph{topological recursion formula} is an inductive
mechanism of defining
a  symmetric $\ell$-form 
$$
W_{g,\ell}(t_L)
=W_{g,\ell}(t_1,\dots,t_\ell)
$$
on $C^\ell$ for any given $g$ and $\ell$ subject to $2g-2+\ell>0$
by
\begin{multline}
\label{eq:recursion}
W_{g,\ell+1}  (t_0,t_L) 
 =-\frac{1}{2\pi i}\oint_{\gamma_\infty}\bigg[
K(t,t_0)
\bigg( W_{g-1, \ell+2} \big(t,s(t),t_L\big) 
\\
 +
\sum_{i=1} ^\ell \Big(
W_{g,\ell}\big(t,t_{L\setminus \{i\}}\big) 
\tensor  B\big(
s(t), t_i\big) + 
W_{g,\ell}\big(s(t),t_{L\setminus \{i\}}\big) 
\tensor  B(t,t_i) \Big)
\\
+ \sum_{g_1+g_2=g, \;
I\sqcup J= L} ^{
\text{stable terms}} 
W_{g_1, |I|+1}
\big( t,t_I\big)
\tensor 
W_{g_2, |J|+1}\big(s(t),t_J\big)\bigg)\bigg]. 
\end{multline}
Here $t_I = (t_i)_{i\in I}$ for a subset $ I\subset L=\{1,2,\dots,\ell\}$,
and   the last sum is taken over all partitions of $g$ and disjoint
decompositions  $I\sqcup J=  L$ subject to the stability 
condition $2g_1 - 1 + |I| >0$ and
$2 g_2 -1 +|J| >0$.
 The  integration is taken with respect to $dt$ on the contour 
 $\gamma_\infty$, which is a 
positively oriented loop in the complex $t$-plane
of large radius so that $|t|>\max(|t_0|, |s(t_0)|)$ for $t\in \gamma_\infty$. 
\end{Def}

 Now we can state the Bouchard-Mari\~no
conjecture, which we prove in Section~\ref{sect:proof}.

\begin{conj}[Bouchard-Mari\~no Conjecture \cite{BM}]
\label{conj:BM}
For every $g$ and $\ell$ subject to the stability condition
$2g-2+\ell >0$, the topological recursion formula 
{\rm{(\ref{eq:recursion})}}
with the initial condition 
\begin{equation}
\label{eq:initial}
\begin{cases}
W_{0,3}(t_1,t_2,t_3) = dt_1\tensor dt_2\tensor dt_3\\
W_{1,1}(t_1) = \frac{1}{24}(t_1-1)(3t_1+1)\;dt_1
\end{cases}
\end{equation}
gives
the Hurwitz differential
$$
W_{g,\ell}(t_1,\dots,t_\ell)= d^{\tensor \ell}\,\hatH_{g,\ell}(t_1,\dots,t_\ell) .
$$
\end{conj}

\begin{rem}
In the literature \cite{BM, EO1}, the topological recursion is
written as
\begin{multline}
\label{eq:recursion2}
W_{g,\ell+1}  (t_0,t_L) 
 =\Res_{q=\bar{q}}\Bigg[
\frac{dE(q,\bar{q},t_0)}{\omega(q,\bar{q})}
\Bigg( W_{g-1, \ell+2} \big(t(q),t(\bar{q}),t_L\big) 
\\
+ \sum_{g_1+g_2=g, \;
I\sqcup J= L} 
W_{g_1, |I|+1}
\big( t(q),t_I\big)
\tensor 
W_{g_2, |J|+1}\big(t(\bar{q}),t_J\big)\Bigg)\Bigg] ,
\end{multline}
including \emph{all} possible terms in the second line, 
with the initial condition
\begin{equation}
\label{eq:EOinitial}
\begin{cases}
W_{0,1}(t_1) = 0\\
W_{0,2}(t_1,t_2) = B(t_1,t_2) .
\end{cases}
\end{equation}
If we single out the stable terms from 
(\ref{eq:recursion2}), then we obtain (\ref{eq:recursion}). 
Although the initial values of $W_{g,\ell}$ given 
in (\ref{eq:EOinitial}) are different from
(\ref{eq:H01}) and (\ref{eq:H02}), the advantage of
(\ref{eq:recursion2}) is to 
be able to 
include (\ref{eq:initial}) as a consequence of the recursion.
\end{rem}

\begin{rem}
It is established in \cite{EO1}
that a solution of the topological recursion is 
a \emph{symmetric} differential form in general.
In our case, 
 the
RHS of the recursion formula (\ref{eq:recursion})
does not appear to be  symmetric in $t_0, t_1,\dots,t_{\ell}$.
We note that our proof of the formula
establishes this symmetry because the Hurwitz differential
is a symmetric polynomial. This situation is again 
strikingly similar to the Mizrakhani recursion \cite{Mir1, Mir2},
where the symmetry appears not as a consequence of 
the recursion, but rather as the geometric nature of the
quantity the recursion calculates, namely, the 
Weil-Petersson volume of the moduli space of bordered 
hyperbolic surfaces.
\end{rem}

\section{Residue calculation}
\label{sect:residue}

In this section we calculate the
residues appearing in the recursion formula
(\ref{eq:recursion}). It turns out to be
equivalent to the direct image operation with respect to
the projection $\pi:C\rightarrow \bC$.

Recall that the kernel $K(t,t_0)$ is a rational expression in terms of
$t, s(t)$ and $t_0$. The function $s(t)$ is an involution 
$s\big(s(t)\big) = t$ defined outside of the slit 
$[0,1]$ of the complex $t$-plane, with logarithmic singularities
at $0$ and $1$. Our idea of computing the residue is to 
decompose the integration over the loop $\gamma_\infty$
into the sum of integrations over $\gamma_\infty - \gamma_{[0,1]}$ and
$\gamma_{[0,1]}$, where $\gamma_{[0,1]}$ is a 
positively oriented thin loop containing
the interval $[0,1]$.

\begin{figure}[htb]
\centerline{\epsfig{file=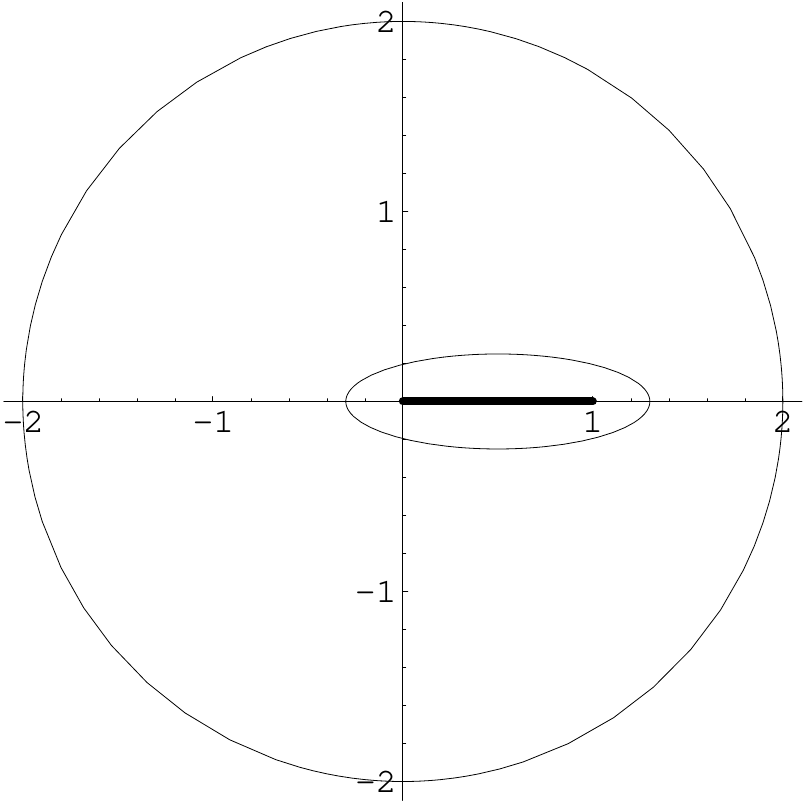, width=2in}}
\caption{The contours of integration. $\gamma_\infty$ is the
 circle of a large radius, and $\gamma_{[0,1]}$ is the thin
loop surrounding the closed interval $[0,1]$. }
\label{fig:contour}
\end{figure}

\begin{Def}
For a Laurent series $\sum_{n\in \mathbb{Z}} a_n t^n$,
we denote
$$
\left[\sum_{n\in \mathbb{Z}} a_n t^n\right]_+ = \sum_{n\ge 0} 
a_n t^n\;.
$$
\end{Def}

\begin{thm}
\label{thm:residuecalculation}
 In terms of the primitives
$\hat{\xi}_n(t)$,  we have
\begin{multline}
\label{eq:rab}
R_{a,b}(t)=
-\frac{1}{2\pi i}
\oint_{\gamma_\infty} K(t',t) \xi_a(t')\xi_b\big(s(t')\big)\\
=
\frac{1}{2}\left[ \frac{ts(t)}{t-s(t)}\Big(
\xi_a(t)\hat{\xi}_{b+1}\big(s(t)\big)+\hat{\xi}_{a+1}\big(s(t)\big)\xi_b(t)
\Big)\right]_+ .
\end{multline}
Similarly, we have
\begin{multline}
\label{eq:rn}
R_{n}(t,t_i)=
-\frac{1}{2\pi i}
\oint_{\gamma_\infty} K(t',t) \Big(\xi_n(t') B\big(s(t'), t_i\big)
+\xi_n\big(s(t')\big)B(t',t_i)\Big)\\
=
\left[\frac{ts(t)}{t-s(t)}\Big(
\hat{\xi}_{n+1}(t)B\big(s(t),t_i\big)+ \hat{\xi}_{n+1}\big(s(t)\big)
B(t,t_i)\Big)
\right]_+ .
\end{multline}
\end{thm}

\begin{proof} 
In terms of the original $z$-coordinate
of \cite{BM}, the residue $R_{a,b}(t)$
 is simply 
the coefficient of $z^{-1}$ in $K(t',t) \xi_a(t')\xi_b\big(s(t')\big)$,
after expanding it in the Laurent series in $z$. Since $\xi_n(t')$ is a
polynomial in $t'=-\frac{1}{z}$, the contribution to the
$z^{-1}$ term in the expression is a polynomial in $t$ because
of the $z$-expansion formula (\ref{eq:kernel-in-z})
for the kernel $K$. Thus we know that $R_{a,b}(t)$ is a polynomial
in $t$.

Let us write $\xi_n(t) = f_n(t)dt$, and 
let $\gamma_{[0,1]}$ be a positively oriented loop containing
the slit $[0,1]$, as in Figure~\ref{fig:contour}.
 On this compact set we have a bound
$$
\left|\frac{ts(t)}{t-s(t)}\;t^2(t-1) s'(t)f_a(t)f_b(s(t))\right|
< M,
$$
since the function is holomorphic outside $[0,1]$. Choose $|t|>>1$.
Then we have
\begin{align*}
&\left|
-\frac{1}{2\pi i}
\oint_{\gamma_{[0,1]}} K(t',t) \xi_a(t')\xi_b\big(s(t')\big)\right|\\
&= \left|
-\frac{1}{2\pi i}
\oint_{\gamma_{[0,1]}} \frac{1}{2} \left(
\frac{1}{t' - t} - \frac{1}{s(t') -t}
\right)   \frac{t's(t')}{s(t')-t'}\;
t'^2(t'-1)s'(t')  f_a(t') f_b\big(s(t')\big)dt'\right|
\tensor dt\\
&<
\frac{M}{2\pi}\oint_{\gamma_{[0,1]}} \frac{1}{2} \left|
\frac{1}{t' - t} - \frac{1}{s(t') -t}
\right|dt'\tensor dt 
\sim \frac{M}{4\pi |t|}dt\;.
\end{align*}
Therefore, 
\begin{align*}
-\frac{1}{2\pi i}
\oint_{\gamma_\infty} K(t',t) \xi_a(t')\xi_b\big(s(t')\big)
&=-\frac{1}{2\pi i}
\oint_{\gamma_\infty-\gamma_{[0,1]}} K(t',t) \xi_a(t')\xi_b\big(s(t')\big)
+{O}(t^{-1})\\
&=
\left[
-\frac{1}{2\pi i}
\oint_{\gamma_\infty-\gamma_{[0,1]}} K(t',t) \xi_a(t')\xi_b\big(s(t')\big)
\right]_+\;.
\end{align*}
Noticing the relation (\ref{eq:stdifferentialrelation}) and the 
fact that $s(t)$ is an involution, we obtain
\begin{align*}
&-\frac{1}{2\pi i}
\oint_{\gamma_\infty-\gamma_{[0,1]}} K(t',t) \xi_a(t')\xi_b\big(s(t')\big)\\
& =
-\frac{1}{2\pi i}
\oint_{\gamma_\infty-\gamma_{[0,1]}} \frac{1}{2} \left(
\frac{1}{t' - t} - \frac{1}{s(t') -t}
\right)   \frac{t's(t')}{s(t')-t'}\;
t'^2(t'-1)s'(t')  f_a(t') f_b\big(s(t')\big)dt'
\tensor dt\\
& =
-\frac{1}{2\pi i}
\oint_{\gamma_\infty-\gamma_{[0,1]}} \frac{1}{2} \cdot
\frac{1}{t' - t}  \cdot \frac{t's(t')}{s(t')-t'}\cdot
t'^2(t'-1)s'(t')  f_a(t') f_b\big(s(t')\big)dt'
\tensor dt\\
& \qquad +
\frac{1}{2\pi i}
\oint_{s(\gamma_\infty)-s(\gamma_{[0,1]})} \frac{1}{2} \cdot
 \frac{1}{s(t') -t}\cdot
   \frac{t's(t')}{s(t')-t'}\cdot
t'^2(t'-1)  f_a(t') f_b\big(s(t')\big)ds(t')
\tensor dt\\
&=
\frac{1}{2} \cdot
\frac{ts(t)}{t-s(t)}\cdot
t^2(t-1)s'(t)  f_a(t) f_b\big(s(t)\big)
 dt\\
&\qquad + 
 \frac{1}{2} \cdot
   \frac{s(t)t}{t-s(t)}\cdot
s(t)^2\big(s(t)-1\big)  f_a\big(s(t)\big) f_b(t)
 dt\\
&=
\frac{ts(t)}{t-s(t)}\;  t^2(t-1)s'(t)\;
\frac{f_a(t)f_b(s(t))+f_a(s(t))f_b(t)}{2} \;dt\\
&=
\half\cdot \frac{ts(t)}{t-s(t)}
\Big(
\xi_a(t)\hat{\xi}_{b+1}\big(s(t)\big) 
+ \hat{\xi}_{a+1}\big(s(t)\big)\xi_b(t)
\Big)\\
&=
\half\cdot \frac{ts(t)}{t-s(t)}
\Big(
\hat{\xi}' _a(t)\hat{\xi}_{b+1}\big(s(t)\big) 
+ \hat{\xi}_{a+1}\big(s(t)\big)\hat{\xi}' _b(t)
\Big) \; dt
\\
&=
\half\cdot \frac{ts(t)}{t-s(t)}
\Big(
\hat{\xi}_{a+1}(t)\hat{\xi}_{b+1}\big(s(t)\big) 
+ \hat{\xi}_{a+1}\big(s(t)\big)\hat{\xi}_{b+1}(t)
\Big) \frac{dt}{t^2(t-1)}.
\end{align*}
Here we used (\ref{eq:hxi-recursion}) and (\ref{eq:stdifferentialrelation})
at the last step.
The proof of the second residue formula is exactly the same. 
\end{proof}

\begin{rem}
The equation for the kernel (\ref{eq:kernelrelation}) implies
\begin{equation*}
R_{a,b}(t)=R_{b,a}(t)= -\left[R_{a,b}\big(s(t)\big)\right]_+.
\end{equation*}
\end{rem}

Let us define polynomials $P_{a,b}(t)$ and $P_n(t,t_i)$ by
\begin{align}
\label{eq:Pab}
&P_{a,b}(t)dt =\half \left[ \frac{ts(t)}{t-s(t)}\;
\frac{dt}{t^2(t-1)}\Big(
\hat{\xi}_{a+1}(t)\hat{\xi}_{b+1}\big(s(t)\big)
+\hat{\xi}_{a+1}\big(s(t)\big)\hat{\xi}_{b+1}(t)
\Big)
\right]_+,\\
\label{eq:Pn}
&P_n(t,t_i) dt\tensor dt_i=d_{t_i}\left[\frac{ts(t)}{t-s(t)}
\left(\frac{\hat{\xi}_{n+1} (t)ds(t)}{s(t)-t_i} 
+ \frac{\hat{\xi}_{n+1}\big(s(t)\big)dt}
{t-t_i}\right)\right]_+.
\end{align}
Obviously $\deg P_{a,b}(t) = 2(a+b+2)$. 
To calculate  $P_n(t,t_i)$, we use
the Laurent series expansion 
\begin{equation}
\label{eq:Laurent}
\frac{1}{t-t_i} = \frac{1}{t}\; \sum_{k= 0} ^\infty
\left(\frac{t_i}{t}\right)^k,
\end{equation}
and take the polynomial part in $t$. We note that it is
automatically a polynomial in $t_i$ as well.
 We thus see that $\deg P_n(t,t_i) = 2n+2$ in
 each variable.

\begin{thm}
The topological recursion formula
{\rm{(\ref{eq:recursion})}}
is equivalent to  the
following 
equation of symmetric differential forms in 
$\ell+1$ variables with polynomial coefficients:
\begin{multline*}
\sum_{n,n_L}\la \tau_n\tau_{n_L} \Lambda_g ^\vee (1) \ra_{g,\ell+1}
 d \hxi_n(t) \tensor
d \hxi_{n_L}(t_L)\\
=
\sum_{i=1} ^\ell
\sum_{m,n_{L\setminus\{i\}}}\la \tau_m\tau_{n_{L\setminus\{i\}}} \Lambda_{g} ^\vee (1) \ra_{g,\ell}
P_{m}(t,t_i)
dt\tensor dt_i\tensor
d \hxi_{n_{L\setminus\{i\}}}(t_{L\setminus\{i\}})\\
+
\Bigg(
\sum_{a,b,n_L}\la \tau_a\tau_b\tau_{n_L} \Lambda_{g-1} ^\vee (1) \ra_{g-1,\ell+2}
\\
+
\sum_{\substack{g_1+g_2=g\\
I\sqcup J=L}} ^{\rm{stable}}
\sum_{\substack{a,n_I\\b,n_J}}\la \tau_a\tau_{n_I}
\Lambda_{g_1} ^\vee (1) \ra_{g_1,|I|+1}\la \tau_b\tau_{n_J}
\Lambda_{g_2} ^\vee (1) \ra_{g_2,|J|+1}
\Bigg)
P_{a,b}(t)dt\tensor  d \hxi_{n_L}(t_L).
\end{multline*}
Here $L=\{1,2\dots,\ell\}$ is an index set, and for a subset 
$I\subset L$, we denote
$$
t_I = (t_i)_{i\in I},\quad
	 n_I = \{\, n_i\,|\, i\in I\,\},\quad 
	\tau_{n_I} = \prod_{i\in I}\tau_{n_i},
	\quad
		d\hxi_{n_I}(t_I) = \bigotimes_{i\in I}
		\frac{d}{dt_i}\hxi_{n_i}(t_i)dt_i.
$$
The last summation in the formula
 is taken over all partitions of $g$ and 
decompositions  of $L$ into 
disjoint subsets $I\sqcup J=  L$ subject to the stability 
condition $2g_1 - 1 + |I| >0$ and
$2 g_2 -1 +|J| >0$.
\end{thm}

\begin{rem}
An immediate
 observation we can make from (\ref{eq:BMalgebraic})
is the simple form of the formula for the case with one marked
point:
\begin{multline}
\label{eq:BMell=1} 
\sum_{n\le 3g-2} \la \tau_n \Lambda_g ^\vee(1)\ra_{g,1}
 \frac{d}{dt}\hxi_n(t)\\
=\sum_{a+b\le 3g-4} \Big(
\la \tau_a\tau_b \Lambda_{g-1} ^\vee(1)\ra_{g-1,2} 
+\sum_{g_1+g_2=g} ^{\text{stable}}\la \tau_a \Lambda_{g_1} ^\vee(1)\ra_{g_1,1}
\la \tau_b\Lambda_{g_2} ^\vee(1)\ra_{g_2,1}
\Big)P_{a,b}(t) .
\end{multline}
\end{rem}

\section{Analysis of the Laplace transforms on the Lambert curve}
\label{sect:Laplace}

As a preparation for Section~\ref{sect:proof} where we
give a proof of (\ref{eq:BMalgebraic}), 
in this section we present analysis tools that
provide the relation among  the Laplace transforms on
the Lambert curve (\ref{eq:Lambert-tw}).
The mystery of the work of Bouchard-Mari\~no \cite{BM}
lies in their $\zeta_n(y)$-forms that play an effective role
in devising the topological recursion for the Hurwitz 
numbers. We have already identified these differential forms as
polynomial forms $d\hxi_n(t)$, where $\hxi_n(t)$'s
are the Lambert $W$-function
and its derivatives.

 Recall Stirling's formula
\begin{multline}
\label{eq:Stirling}
\log \Gamma(z) 
=  \frac{1}{2}\log 2\pi
+\left(z-\half\right) \log z - z  
\\
+\sum_{r=1} ^m \frac{B_{2r}}{2r(2r-1)} z^{-2r+1}
-\frac{1}{2m} \int_0 ^\infty 
\frac{B_{2m}(x-[x])}{(z+x)^{2m}}dx,
\end{multline}
where $m$ is an arbitrary cut-off parameter, 
$B_{r}(s)$ is the Bernoulli polynomial defined by
$$
\frac{z e^{zx}}{e^z-1} = \sum_{n=0} ^\infty B_r(x) \;\frac{z^r}
{r!},
$$
$B_r = B_r(0)$ is the Bernoulli number,
 and   $[x]$ is the largest integer
not exceeding $x\in \bR$.
For  $N>0$, we have 
\begin{align*}
e^{-N}\frac{N^{N+n}}{N!} 
&= \frac{1}{\sqrt{2\pi}} N^{n-\frac{1}{2}}
\exp\left(
-\sum_{r=1} ^m \frac{B_{2r}}{2r(2r-1)} N^{-2r+1}
\right)
\exp\left(
\frac{1}{2m} \int_0 ^\infty 
\frac{B_{2m}(x-[x])}{(N+x)^{2m}}dx
\right).
\end{align*}
Let us define the coefficients $s_k$ for $k\ge 0$ by
\begin{multline}
\label{eq:sk}
\sum_{k=0} ^\infty s_k N^{-k} = 
\exp\left(
-\sum_{r=1} ^\infty \frac{B_{2r}}{2r(2r-1)} N^{-2r+1}
\right)\\
= 
1-\frac{1}{12} N^{-1} + \frac{1}{288} N^{-2} + \frac{139}{51840}
N^{-3}-\frac{571}{2488320} N^{-4}+\cdots .
\end{multline}
Then for a large $N$ 
we have an  asymptotic expansion
\begin{equation}
\label{eq:N-asymptotic}
e^{-N}\frac{N^{N+n}}{N!} \sim
\frac{1}{\sqrt{2\pi}}  N^{n-\frac{1}{2}} \sum_{k=0} ^\infty
s_k N^{-k}.
\end{equation}

 \begin{Def}
Let us introduce an infinite sequence of Laurent series
\begin{equation}
\label{eq:eta_n}
\eta_n(v) = 
\frac{1}{v}\;
\sum_{k=0} ^\infty 
s_k \;\frac{\big(2(n-k)-1\big)!!}{v^{2(n-k)}} 
=-\eta_n(-v) 
\end{equation}
for every $n\in \bZ$, where $s_k$'s are the coefficients
defined in {\rm{(\ref{eq:sk})}}.
\end{Def}

The following lemma relates the polynomial forms
$\xi_n(t)=d\hxi_n(t)$, the functions $\eta_n(v)$, and 
the Laplace transform.

\begin{prop} 
\label{prop:LT=eta}
For $n\ge 0$, we have 
\begin{equation}
\label{eq:LT=eta}
 \int_{0} ^\infty
e^{-s}\frac{s^{s+n}}{\Gamma(s+1)} e^{-sw}ds 
=
 \eta_n(v) +\const + O(w)
\end{equation}
with respect to the choice of the branch of $\sqrt{w}$ specified 
by $v=-\sqrt{2w}$ as in
{\rm{(\ref{eq:w-in-t})}} and 
{\rm{(\ref{eq:v(t)})}}, 
where $O(w)$ denotes a holomorphic function in $w = \half v^2$ 
defined around $w=0$ which vanishes at $w=0$. 
The
substitution of {\rm{(\ref{eq:v(t)})}}
 in $\eta_n(v)$
yields
\begin{equation}
\label{eq:eta=xi}
\eta_n (v) 
=
\half\bigg(\eta_n(v)-\eta_n(-v)\bigg)
=
\half\bigg(
\hat{\xi}_n(t)- \hat{\xi}_n\big(s(t)\big) \bigg),
\end{equation}
where $s(t)$ is the involution  of {\rm{(\ref{eq:s})}}. 
This formula is valid for $n\ge -1$, and in particular, 
we have a relation between the kernel and $\eta_{-1}(v)$:
\begin{equation}
\label{eq:eta-1}
\eta_{-1}(v) = \half\left(
\hat{\xi}_{-1}(t)-\hat{\xi}_{-1}\big(s(t)\big)
\right)
=\half\;\frac{t-s(t)}{ts(t)}.
\end{equation}
More precisely, for $n\ge -1$, we have
\begin{equation}
\label{eq:eta=xi+F}
\begin{cases}
\eta_{n}(v) =\hat{\xi}_{n}(t) +F_n(w)\\
\eta_{n}(-v) =\hat{\xi}_{n}\big(s(t)\big) +F_n(w)
\end{cases},
\end{equation}
where $F_n(w)$ is a holomorphic function in $w$. 
\end{prop}

\begin{proof} 
From definition (\ref{eq:eta_n}),  it is obvious that 
the series  $\eta_n(v)$  satisfies the recursion relation
\begin{equation}
\label{eq:eta-recursion}
\eta_{n+1}(v) = -\frac{1}{v}\;\frac{d}{dv}\eta_n(v)
\end{equation}
for all $n\in\bZ$.
The integral (\ref{eq:LT=eta}) also satisfies
the same recursion for $n\ge 0$.  
So choose an $n\ge 0$. We have
 an estimate
$$
e^{-s}\frac{s^{s+n}}{\Gamma(s+1)}
=
\frac{1}{\sqrt{2\pi}} s^{n-\half} \sum_{k=0} ^n s_k s^{-k}
+ O(s^{-\frac{3}{2}})
$$
that is valid for $s>1$. Since the integral
\begin{equation*}
 \int_{0} ^1
e^{-s}\frac{s^{s+n}}{\Gamma(s+1)} e^{-sw}ds 
\end{equation*}
is an entire function in $w$, we have
\begin{align*}
&\int_{0} ^\infty
e^{-s}\frac{s^{s+n}}{\Gamma(s+1)} e^{-sw}ds 
=
 \int_{0} ^1
e^{-s}\frac{s^{s+n}}{\Gamma(s+1)} e^{-sw}ds 
+
 \int_{1} ^\infty
e^{-s}\frac{s^{s+n}}{\Gamma(s+1)} e^{-sw}ds 
\\
&=
 \int_{0} ^\infty \bigg(
\frac{1}{\sqrt{2\pi}} s^{n-\half} \sum_{k=0} ^n s_k s^{-k}
\bigg)
e^{-sw} ds
+ 
 \int_{1} ^\infty
O(s^{-\frac{3}{2}}) e^{-sw}ds  +\const + O(w)
\\
&=
\frac{1}{v}\;
\sum_{k=0} ^n 
s_k \;\frac{\big(2(n-k)-1\big)!!}{v^{2(n-k)}} + 
\const + vO(w) + O(w).
\end{align*}
This formula is valid for all $n\ge 0$. Starting it 
from a large  $n>>0$ and using the recursion 
(\ref{eq:eta-recursion}) backwards, we conclude that
the $vO(w)$ terms in the above formula are indeed the positive
power terms of $\eta_n(v)$. 
The principal part of $\eta_n(v)$ does not depend on the
addition of positive power terms in $w=\half v^2$, 
since $-\frac{1}{v}\frac{d}{dv}$ transforms
a positive even power of $v$ to a non-negative even power
and does not create any negative powers.
This proves (\ref{eq:LT=eta}).

Next let us estimate the holomorphic error term
$O(w)$ in (\ref{eq:LT=eta}).
When $n\le -1$, the Laplace transform (\ref{eq:LT=eta})
does not converge. However, the truncated  integral
$$
 \int_{1} ^\infty
e^{-s}\frac{s^{s+n}}{\Gamma(s+1)} e^{-sw}ds 
$$
always converges and defines a
holomorphic function in $v = - \sqrt{2w}$, which still satisfies 
the recursion relation (\ref{eq:eta-recursion}). 
Again by the inverse induction, we have
\begin{equation}
\label{eq:truncatedLT}
 \int_{1} ^\infty
e^{-s}\frac{s^{s+n}}{\Gamma(s+1)} e^{-sw}ds 
= \eta_n(v) + O(w)
\end{equation}
for every $n<0$.
Now by the Euler summation formula, for $n\le- 1$
and  $Re(w)>0$, we have
\begin{multline}
\label{eq:Euler}
\int_{1} ^\infty
e^{-s}\frac{s^{s+n}}{\Gamma(s+1)} e^{-sw}ds 
-\sum_{k=2} ^\infty e^{-k}\frac{k^{k+n}}{k!}e^{-kw}
\\
=
-\half e^{-(w+1)}
+
 \int_{1} ^\infty \left(s-[s]-\half\right)
 \frac{d}{ds}\left(
e^{-s}\frac{s^{s+n}}{\Gamma(s+1)} e^{-sw}\right)ds .
\end{multline}
Note that the RHS of (\ref{eq:Euler}) is holomorphic in $w$
around $w=0$.
From (\ref{eq:xihat}), (\ref{eq:truncatedLT}) and 
(\ref{eq:Euler}), we establish a comparison
formula 
\begin{equation}
\label{eq:eta=xi+F-1}
\eta_{-1}(v) -\hat{\xi}_{-1}(t) = F_{-1}(w),
\end{equation}
where $F_{-1}(w)$ is a holomorphic function in $w$
defined near $w=0$, and
we identify the coordinates $t, v$ and $w$ by
the relations  (\ref{eq:tvw}) and (\ref{eq:v(t)}).
Note that the relation  (\ref{eq:tvw}) is invariant
under the involution 
\begin{equation}
\label{eq:vt-involution}
\begin{cases}
v\longmapsto -v\\
t\longmapsto s(t) 
\end{cases} .
\end{equation}
Therefore, we also have  
\begin{equation*}
\eta_{-1}(-v) -\hat{\xi}_{-1}\big(s(t)\big) = F_{-1}(w).
\end{equation*}
Thus we obtain
\begin{equation*}
\eta_{-1}(v) =\half
\bigg(
\hat{\xi}_{-1}(t)- \hat{\xi}_{-1}\big(s(t)\big) \bigg),
\end{equation*}
which proves  (\ref{eq:eta-1}).
Since 
$-\frac{1}{v}\frac{d}{dv} = t^2(t-1)\frac{d}{dt}$,
the recursion relations (\ref{eq:eta-recursion}) and 
 (\ref{eq:hxi-recursion}) for
$\hat{\xi}_n(t)$ are exactly the same. 
We note that from (\ref{eq:stdifferentialrelation}) we have
$$
-\frac{1}{v}\frac{d}{dv} = t^2(t-1)\frac{d}{dt}
= s(t)^2 \big(s(t)-1\big) \frac{d}{ds(t)}.
$$
Therefore, the difference 
$\hat{\xi}_{n+1}(t)- \hat{\xi}_{n+1}\big(s(t)\big)$
satisfy the same recursion
\begin{equation}
\label{eq:xi-xi-recursion}
\hat{\xi}_{n+1}(t)- \hat{\xi}_{n+1}\big(s(t)\big) 
=
 t^2(t-1)\frac{d}{dt}
\bigg(
\hat{\xi}_n(t)- \hat{\xi}_n\big(s(t)\big) \bigg).
\end{equation}
The recursions (\ref{eq:eta-recursion}) and 
(\ref{eq:xi-xi-recursion}), together with the 
initial condition (\ref{eq:eta-1}), establish
(\ref{eq:eta=xi}).
Application of  the differential operator
$$
-\frac{1}{v}\;\frac{d}{dv} = -\frac{d}{dw} = 
t^2(t-1)\frac{d}{dt}
$$
$(n+1)$-times to  (\ref{eq:eta=xi+F-1}) yields
\begin{equation*}
\eta_{n}(v) -\hat{\xi}_{n}(t) = F_{n}(w),
\end{equation*}
where $F_n(w) = (-1)^n \frac{d^n}{dw^n}F_{-1}(w)$ is
a holomorphic function in $w$ around $w=0$. Involution
(\ref{eq:vt-involution}) then gives
\begin{equation*}
\eta_{n}(-v) -\hat{\xi}_{n}\big(s(t)\big) = F_{n}(w).
\end{equation*}
This completes the proof of the proposition.
\end{proof}

As we have noted in Section~\ref{sect:residue},
the residue calculations appearing in the Bouchard-Mari\~no
recursion formula (\ref{eq:recursion}) are essentially
evaluations of the
product of $\xi$-forms at the point $t$ and $s(t)$ on the
Lambert curve, if we truncate the result to the polynomial
part. In terms of the $v$-coordinate, these two 
points correspond to $v$ and $-v$. Thus we have

\begin{cor}
\label{cor:Pab}
The residue polynomials of {\rm{(\ref{eq:Pab})}}
are given by
\begin{multline}
\label{eq:Pab-in-eta}
P_{a,b}(t)dt
=
\half
 \left[ \frac{ts(t)}{t-s(t)}\;\frac{dt}{t^2(t-1)}
 \Big(
\hat{\xi}_{a+1}(t)\hat{\xi}_{b+1}\big(s(t)\big)
+\hat{\xi}_{a+1}\big(s(t)\big)\hat{\xi}_{b+1}(t)
\Big)
\right]_+\\
=
\half \left[\left.\frac{
\eta_{a+1}(v)\eta_{b+1}(v)}
{\eta_{-1}(v)}\;vdv\right|_{v=v(t)}
\right]_+,
\end{multline}
where 
the reciprocal of 
\begin{equation*}
\eta_{-1}(v) {=} 
\sum_{k=0} ^\infty 
s_k \big(2(-1-k)-1\big)!! \; v^{2k+1}
=-v\left(1+\sum_{k=1} ^\infty (-1)^{k}s_k \frac{1}{(2k+1)!!}
\;v^{2k}\right) 
\end{equation*}
is defined by
$$
\frac{1}{\eta_{-1}(v)} =- \frac{1}{v}\left(1+
\sum_{m=0} ^\infty 
\left(\sum_{k=1} ^\infty (-1)^{k-1}s_k \frac{1}{(2k+1)!!}
\;v^{2k}\right)^m
\right).
$$
\end{cor}

\begin{proof}
Using the formulas established in Proposition~\ref{prop:LT=eta},
we compute
\begin{align*}
&\half\;\frac{ts(t)}{t-s(t)}\;\frac{dt}{t^2(t-1)}
 \Big(
\hat{\xi}_{a+1}(t)\hat{\xi}_{b+1}\big(s(t)\big)
+\hat{\xi}_{a+1}\big(s(t)\big)\hat{\xi}_{b+1}(t)
\Big)\\
&=
-\frac{ts(t)}{t-s(t)}\;\frac{dt}{t^2(t-1)}
\bigg(\frac{
\hat{\xi}_{a+1}(t)-\hat{\xi}_{a+1}\big(s(t)\big)}{2}\;
\frac{
\hat{\xi}_{b+1}(t)-\hat{\xi}_{b+1}\big(s(t)\big)}{2}\\
&\qquad\qquad 
-
\frac{
\hat{\xi}_{a+1}(t)+\hat{\xi}_{a+1}\big(s(t)\big)}{2}\;
\frac{
\hat{\xi}_{b+1}(t)+\hat{\xi}_{b+1}\big(s(t)\big)}{2}
\bigg)
\\
&=
-\frac{1}{2\eta_{-1}(v)}\bigg(\eta_{a+1}(v)\eta_{b+1}(v)
-F_{a+1}(w)F_{b+1}(w)\bigg)(-v)dv
\\
&=
\frac{\eta_{a+1}(v)\eta_{b+1}(v)}{2\eta_{-1}(v)}vdv
+ \big(\const + O(w)\big)dv .
\end{align*}
From  (\ref{eq:v(t)}) we see $\big[\big(\const + O(w)\big)dv
|_{v\mapsto t}\big]_+=0$. This completes the proof of 
(\ref{eq:Pab-in-eta}).
\end{proof}

For the terms involving the Cauchy differentiation kernel
$B(t_i,t_j)$, 
we have the following formula.

\begin{prop}
\label{prop:Pn}
As a polynomial in $t$ and $t_j$, we have the following
equality:
\begin{multline}
\label{eq:Pn-in-eta}
P_n(t,t_j) dt\tensor dt_j
=
d_{t_j}\left[\frac{ts(t)}{t-s(t)}
\left(\frac{\hat{\xi}_{n+1} (t)ds(t)}{s(t)-t_j} 
+ \frac{\hat{\xi}_{n+1}\big(s(t)\big)dt}
{t-t_j}\right)\right]_+
\\
=d_{t_j}\left[\left.
\frac{\eta_{n+1}(v_j)}{\eta_{-1}(v)}\cdot
\frac{1}{v^2} \sum_{m=0} ^{\rm{finite}}
\left(\frac{v_j}{v}\right)^{2m}
\; vdv\right|_{\substack{v=v(t)\\
v_j=v(t_j)}}
\right]_+ .
\end{multline}
In the RHS we first evaluate the expression 
at $v=v(t)$ and $v_j=v(t_j)$, then expand it as a 
series in $\frac{1}{t}$ and 
$\frac{1}{t_j}$, and  finally truncate it  as a polynomial in both
$t$ and $t_j$.
\end{prop}

\begin{proof}
From the formulas for $\hat{\xi}_n(t)$ and $\eta_n(v)$,
we know that both expressions have the same degree
$2n+2$ in $t$ and $t_j$.
Since the powers of $v_j$ in the summation
 $\sum_{m=0} ^{\rm{finite}}
\left(\frac{v_j}{v}\right)^{2m}$ is non-negative, clearly
we have
$$
d_{t_j}\left[\left.
\frac{\eta_{n+1}(v)}{\eta_{-1}(v)}\cdot
\frac{1}{v^2} \sum_{m=0} ^{\rm{finite}}
\left(\frac{v_j}{v}\right)^{2m}
\; vdv\right|_{\substack{v=v(t)\\
v_j=v(t_j)}}
\right]_+ =0.
$$
Thus we can replace the RHS of (\ref{eq:Pn-in-eta})
by
$$
d_{t_j}\left[\left.
\frac{\eta_{n+1}(v)-\eta_{n+1}(v_j)}{\eta_{-1}(v)}\cdot
\frac{1}{v^2} \sum_{m=0} ^{\rm{finite}}
\left(\frac{v_j}{v}\right)^{2m}
\; (-v)dv\right|_{\substack{v=v(t)\\
v_j=v(t_j)}}
\right]_+ .
$$
Since the degree of $\eta_{n+1}\big(v(t_j)\big)$ in $t_j$ is
$2n+3$, the finite sum in $m$ of the above 
expression contributes nothing
for $m>n+2$. Therefore,
\begin{align*}
&
d_{t_j}\left[\left.
\frac{\eta_{n+1}(v)-\eta_{n+1}(v_j)}{\eta_{-1}(v)}\cdot
\frac{1}{v^2} \sum_{m=0} ^{\rm{finite}}
\left(\frac{v_j}{v}\right)^{2m}
\; (-v)dv\right|_{\substack{v=v(t)\\
v_j=v(t_j)}}
\right]_+
 \\
= \; 
&
d_{t_j}\left[\left.
\frac{t s(t)}{t-s(t)}
\bigg(
\frac{\hat{\xi}_{n+1}(t)-\hat{\xi}_{n+1}(t_j)}{w-w_j}
+\frac{F_{n+1}(w) - F_{n+1}(w_j)}{w-w_j}
\bigg)
(-dw)\right|_{\substack{v=v(t)\\
v_j=v(t_j)}}
\right]_+ 
 \\
= \; 
&
d_{t_j}\left[\left.
\frac{t s(t)}{t-s(t)}
\;
\frac{\hat{\xi}_{n+1}(t)-\hat{\xi}_{n+1}(t_j)}{w-w_j}
\;
(-dw)\right|_{\substack{v=v(t)\\
v_j=v(t_j)}}
\right]_+ 
\end{align*}
because of (\ref{eq:eta=xi+F}). We also used the fact that
$$
\frac{1}{v^2} \sum_{m=0} ^{\rm{finite}}
\left(\frac{v_j}{v}\right)^{2m}
\; vdv 
=\half\; \frac{dw}{w-w_j} + O(w_j ^{n+2})dw,
$$
and that
$\frac{F_{n+1}(w) - F_{n+1}(w_j)}{w-w_j}$ is 
holomorphic along $w=w_j$.
Let us use once again $-\hat{\xi}_{n+1}(t) = 
\hat{\xi}_{n+1}\big(s(t)\big) +2F_{n+1}(w)$
and 
$
-\frac{dw}{w-w_j}=
-\frac{2vdv}{v^2-v_j ^2} =\left(
 \frac{1}{-v-v_j}-\frac{1}{v-v_j}\right)dv.
$
We obtain
\begin{multline*}
d_{t_j}\left[\left.
\frac{t s(t)}{t-s(t)}
\;
\frac{\hat{\xi}_{n+1}(t)-\hat{\xi}_{n+1}(t_j)}{w-w_j}
\;
(-dw)\right|_{\substack{v=v(t)\\
v_j=v(t_j)}}
\right]_+
 \\
= 
d_{t_j}\left[\left.
\frac{t s(t)}{t-s(t)}
\left(
\frac{\hat{\xi}_{n+1}(t)-\hat{\xi}_{n+1}(t_j)}{-v-v_j}\;
\frac{dv}{ds(t)}ds(t)\right)
\right|_{\substack{v=v(t)\\
v_j=v(t_j)}}
\right]_+ 
\\
+
d_{t_j}\left[\left.
\frac{t s(t)}{t-s(t)}
\left(
\frac{\hat{\xi}_{n+1}\big(s(t)\big)-\hat{\xi}_{n+1}\big(s(t_j)\big)}{v-v_j}\;
\frac{dv}{dt}dt\right)
\right|_{\substack{v=v(t)\\
v_j=v(t_j)}}
\right]_+ 
 \\
= 
d_{t_j}\left[
\frac{t s(t)}{t-s(t)}
\left(
\frac{\hat{\xi}_{n+1}(t)-\hat{\xi}_{n+1}(t_j)}{s(t)-t_j}\;
\frac{s(t)-t_j}{v\big(s(t)\big) - v(t_j)}\;
\frac{dv\big(s(t)\big)}{ds(t)}ds(t)
\right)
\right]_+ 
\\
+
d_{t_j}\left[
\frac{t s(t)}{t-s(t)}
\left(
\frac{\hat{\xi}_{n+1}\big(s(t)\big)-\hat{\xi}_{n+1}\big(s(t_j)\big)}{t-t_j}\;
\frac{t-t_j}{v(t)-v(t_j)}\;
\frac{dv(t)}{dt}dt\right)
\right]_+ .
\end{multline*}
Here we remark that
\begin{multline*}
d_{t_j}\left[
\frac{t s(t)}{t-s(t)}
\left(
\frac{\hat{\xi}_{n+1}(t)-\hat{\xi}_{n+1}(t_j)}{s(t)-t_j}ds(t)
+\frac{\hat{\xi}_{n+1}\big(s(t)\big)-\hat{\xi}_{n+1}\big(s(t_j)\big)}{t-t_j}dt
\right)
\right]_+ 
\\
=
d_{t_j}\left[
\frac{t s(t)}{t-s(t)}
\left(
\frac{\hat{\xi}_{n+1}(t)}{s(t)-t_j}ds(t)
+\frac{\hat{\xi}_{n+1}\big(s(t)\big)}{t-t_j}dt
\right)
\right]_+ ,
\end{multline*}
because the extra terms in the LHS do not contribute 
to the polynomial part in $t$. Therefore, it suffices to show
that 
\begin{multline}
\label{eq:sufficient}
d_{t_j}\left[
\frac{t s(t)}{t-s(t)}\;
\frac{\hat{\xi}_{n+1}(t)-\hat{\xi}_{n+1}(t_j)}{s(t)-t_j}
\left(\frac{s(t)-t_j}{v\big(s(t)\big) - v(t_j)}\;
\frac{dv\big(s(t)\big)}{ds(t)}-1
\right)ds(t)
\right]_+ 
\\
+
d_{t_j}\left[
\frac{t s(t)}{t-s(t)}
\;
\frac{\hat{\xi}_{n+1}\big(s(t)\big)-\hat{\xi}_{n+1}\big(s(t_j)\big)}{t-t_j}\left(
\frac{t-t_j}{v(t)-v(t_j)}\;
\frac{dv(t)}{dt}-1\right)dt
\right]_+ 
\\
=
d_{t_j}\left[
\frac{t s(t)}{t-s(t)}\;
\bigg(\hat{\xi}_{n+1}(t)-\hat{\xi}_{n+1}(t_j)\bigg)
\left(\frac{-dv(t)}{-v(t) - v(t_j)}-
\frac{ds(t)}{s(t)-t_j}
\right)
\right]_+ 
\\
+
d_{t_j}\left[
\frac{t s(t)}{t-s(t)}
\;
\bigg(\hat{\xi}_{n+1}\big(s(t)\big)-\hat{\xi}_{n+1}\big(s(t_j)\big)
\bigg)\left(
\frac{dv(t)}{v(t)-v(t_j)}-
\frac{dt}{t-t_j}\right)
\right]_+ 
\\
=
d_{t_j}\left[
\frac{t s(t)}{t-s(t)}\;
\bigg(\hat{\xi}_{n+1}(t)-\hat{\xi}_{n+1}(t_j)\bigg)
\left(\frac{-dv(t)}{-v(t) - v(t_j)}-
\frac{ds(t)}{s(t)-t_j}
\right)
\right]_+ 
\\
-
d_{t_j}\left[
\frac{t s(t)}{t-s(t)}
\;
\bigg(\hat{\xi}_{n+1}(t)-\hat{\xi}_{n+1}(t_j)
\bigg)\left(
\frac{dv(t)}{v(t)-v(t_j)}-
\frac{dt}{t-t_j}\right)
\right]_+ 
=0,
\end{multline}
in light of 
 (\ref{eq:eta=xi+F}).
At this stage we need the following Lemma:

\begin{lem} For every $n\ge 0$ we have the identity
\begin{equation}
0= d_{t_j}\left[\left.
(t^n-t_j ^n)
\left(\frac{-dv}{-v-v_j}
-\frac{ds(t)}{s(t)-t_j}-\frac{dv}{v-v_j}+\frac{dt}{t-t_j}
\right)
\right|_{\substack{v=v(t)\\
v_j=v(t_j)}}
\right]_+ .
\end{equation}
\end{lem}

\begin{proof}[Proof of Lemma]
First let us recall that
$B(t,t_j) = d_{t_j}\left(\frac{dt}{t-t_j}\right)$ is the 
Cauchy differentiation kernel of the Lambert curve $C$, which is a
symmetric quadratic form on $C\times C$ with 
second order poles along the diagonal $t=t_j$. 
The function $v=v(t)$ is a local coordinate change,
which transforms $v=0$ to $t=\infty$. Therefore, the form
$\frac{dv}{v-v_j}-\frac{dt}{t-t_j}$ is a 
meromorphic $1$-form locally defined on $C\times C$, 
which is actually holomorphic on a neighborhood of  the diagonal
and vanishes on the diagonal. Therefore, 
it has the Taylor series expansion in $\frac{1}{t}$ and 
$\frac{1}{t_j}$ without a constant term. 

Since $v\big(s(t)\big) = -v(t)$, the form
$\frac{dv}{v+v_j}
-\frac{ds(t)}{s(t)-t_j}$ is the pull-back of 
$\frac{dv}{v-v_j}-\frac{dt}{t-t_j}$ via the local involution 
$s:C\rightarrow C$ that is applied to the first factor. 
Thus this is again a local holomorphic $1$-form on $C\times C$
and has exactly the same  Taylor expansion in 
$\frac{1}{s(t)}$ and 
$\frac{1}{t_j}$. Therefore, in the $\frac{1}{t_j}$-expansion of 
the difference 
$\frac{-dv}{-v-v_j}
-\frac{ds(t)}{s(t)-t_j}-\frac{dv}{v-v_j}+\frac{dt}{t-t_j}$,
each coefficient does not contain a constant term because
it is cancelled by taking the difference. It implies that
the difference $1$-form
does not contain any terms without $\frac{1}{t}$.
In other words, we have
\begin{equation}
\label{eq:tj-terms}
0= \left[\left.
t_j ^n
\left(\frac{-dv}{-v-v_j}
-\frac{ds(t)}{s(t)-t_j}-\frac{dv}{v-v_j}+\frac{dt}{t-t_j}
\right)
\right|_{\substack{v=v(t)\\
v_j=v(t_j)}}
\right]_+ .
\end{equation}
Note that  we have an expression of the form
\begin{equation}
\label{eq:ttj-expansion}
\left[
\frac{-dv}{-v-v_j}
-\frac{ds(t)}{s(t)-t_j}-\frac{dv}{v-v_j}+\frac{dt}{t-t_j}
\right]
_{\substack{v=v(t)\\
v_j=v(t_j)}}
=f(t^{-1}) + \frac{1}{t_j}\;F\left(\frac{1}{t},\frac{1}{t_j}\right),
\end{equation}
where $f$ is a power series in one variable and $F$ a power 
series in two variables. Therefore, 
\begin{equation}
\label{eq:t-terms}
0= d_{t_j}\left[\left.
t ^n
\left(\frac{-dv}{-v-v_j}
-\frac{ds(t)}{s(t)-t_j}-\frac{dv}{v-v_j}+\frac{dt}{t-t_j}
\right)
\right|_{\substack{v=v(t)\\
v_j=v(t_j)}}
\right]_+ .
\end{equation}
Lemma follows from (\ref{eq:tj-terms}) and (\ref{eq:t-terms}). 
\end{proof}

It is obvious from (\ref{eq:tj-terms}) and 
(\ref{eq:ttj-expansion}) that
\begin{equation}
\label{eq:n-expression}
0= d_{t_j}\left[\left.
(t^n-t_j ^n)\frac{ts(t)}{t-s(t)}
\left(\frac{-dv}{-v-v_j}
-\frac{ds(t)}{s(t)-t_j}-\frac{dv}{v-v_j}+\frac{dt}{t-t_j}
\right)
\right|_{\substack{v=v(t)\\
v_j=v(t_j)}}
\right]_+ .
\end{equation}
Since $\hat{\xi}_{n+1}(t)$ is a polynomial in $t$,
(\ref{eq:sufficient}) follows from (\ref{eq:n-expression}).
This completes the proof of the proposition.  
\end{proof}

\section{Proof of the Bouchard-Mari\~no topological recursion formula}
\label{sect:proof}

In this section we prove (\ref{eq:BMalgebraic}).
Since it is equivalent to Conjecture~\ref{conj:BM},
we establish the Bouchard-Mari\~no conjecture. 
Our procedure is to take the direct image of
the 
equation (\ref{eq:cajLT}) on the Lambert curve
via the projection
$\pi:C\rightarrow \bC$.
To compute the direct image, it is easier to switch to 
the coordinate $v$ of the Lambert curve, because
of the relation (\ref{eq:v(s)}).  
This simple relation tells us that the direct image 
of a function $f(v)$ 
on $C$ via the projection $\pi:C\rightarrow
\bC$ is
just the even powers of the $v$-variable in $f(v)$:
$$
\pi_* f =  f(v) + f(-v).
$$ 
After taking the direct image, we extract the
principal part of the meromorphic function in $v$,
which becomes the Bouchard-Mari\~no recursion
(\ref{eq:BMalgebraic}).
To this end, we utilize the formulas
developed in Section~\ref{sect:Laplace}.

Here again let us consider the $\ell=1$ case first.
We start with Proposition~\ref{prop:ell=1LT}.

\begin{thm}
\label{thm:ell=1}
We have the following equation 
\begin{multline}
\label{eq:ell=1-inv}
-\sum_{n\le 3g-2}\la \tau_n \Lam_g ^\vee (1)\ra_{g,1}
\eta_{-1}(v)\eta_{n+1}(v)
=
\half
\sum_{a+b\le 3g-4}
\bigg[
\la \tau_a\tau_b \Lam_{g-1} ^\vee (1)\ra_{g-1,2}
\\
+\sum_{g_1+g_2=g} ^{\rm{stable}}
\la \tau_a\Lam_{g_1} ^\vee(1) \ra_{g_1,1}
\la \tau_b\Lam_{g_2} ^\vee(1) \ra_{g_2,1}
\bigg]
\bigg(
\eta_{a+1}(v)\eta_{b+1}(v)+O_w(1)
\bigg),
\end{multline}
where $O_w(1)$ denotes a holomorphic function in $w=\half v^2$.
\end{thm}

\begin{proof}
We use (\ref{eq:eta=xi+F}) to change from the $t$-variables
to the $v$-variables.
The function factor of the LHS of (\ref{eq:ell=1LT}) becomes
\begin{equation*}
(2g-1)\eta_n(v)+\eta_{n+1}(v) -\eta_{-1}(v)\eta_{n+1}(v)
+vf_L(w) + \const +O(w)
,
\end{equation*}
where $f_L(w)$ is a Laurent series in $w$. 
The function factor of the RHS is
\begin{equation*}
\eta_{a+1}(v)\eta_{b+1}(v)+vf_R(w) + \const +O(w),
\end{equation*}
where $f_R(w)$ is another Laurent series in $w$. 
We note that the product of two $\eta_n$-functions is a
Laurent series in $w$. Therefore, extracting the
principal part of the Laurent series in $w$, we obtain \begin{multline*}
-\sum_{n\le 3g-2}\la \tau_n \Lam_g ^\vee (1)\ra_{g,1}
\eta_{-1}(v)\eta_{n+1}(v)
=
\half
\sum_{a+b\le 3g-4}
\bigg[
\la \tau_a\tau_b \Lam_{g-1} ^\vee (1)\ra_{g-1,2}
\\
+\sum_{g_1+g_2=g} ^{\rm{stable}}
\la \tau_a\Lam_{g_1} ^\vee(1) \ra_{g_1,1}
\la \tau_b\Lam_{g_2} ^\vee(1) \ra_{g_2,1}
\bigg]
\bigg(
\eta_{a+1}(v)\eta_{b+1}(v)+\const +O(w)
\bigg),
\end{multline*}
which completes the proof of the theorem.
\end{proof}

\begin{cor}
\label{cor:ell=1}
The cut-and-join equation {\rm{(\ref{eq:cajcoefficient})}}
for the case of $\ell=1$ implies the topological 
recursion {\rm{(\ref{eq:BMell=1})}}.
\end{cor}

\begin{proof}
Going back to the $t$-coordinates and
using (\ref{eq:xihat}), (\ref{eq:hxi-recursion}), and 
(\ref{eq:eta=xi+F}) in (\ref{eq:ell=1-inv}), we establish
\begin{multline}
\label{eq:ell=1final}
\sum_{n\le 3g-2} \la \tau_n\Lambda_g ^\vee(1)
\ra_{g,1}\xi_n(t)
=
\half
\sum_{a+b\le 3g-4}
\Bigg(
 \la \tau_a\tau_b\Lambda_{g-1} ^\vee(1)
\ra_{g-1,2} \\
+
\sum_{g_1+g_2=g} ^{\text{stable}}
\la \tau_a\Lambda_{g_1} ^\vee(1)
\ra_{g_1,1}
\la \tau_b\Lambda_{g_2} ^\vee(1)
\ra_{g_2,1}
\Bigg)
\left[
\left.\frac{
\eta_{a+1}(v)\eta_{b+1}(v)}
{\eta_{-1}(v)} \;vdv\right|_{v=v(t)}
\right]_+  ,
\end{multline}
since 
$$
\left[\left.\frac{
\const +O(w)}
{\eta_{-1}(v)} \;vdv\right|_{v=v(t)}
\right]_+ =0.
$$
From Corollary~\ref{cor:Pab}, 
we conclude that (\ref{eq:ell=1final}) is identical to 
(\ref{eq:BMell=1}). This completes the proof of the
topological recursion for $\ell=1$.
\end{proof}

We are now ready to give a proof of
(\ref{eq:BMalgebraic}). The starting point is
the Laplace transform of the cut-and-join
equation, as we have established in Theorem~\ref{thm:CAJLT}.
Since we are interested in the
\emph{principal part} of the formula in 
the $v$-coordinate expansion, in what follows we
ignore all terms that contain any positive powers of 
one of the $v_i$'s.

First let us deal with the unstable $(0,2)$-terms
computed in (\ref{eq:H02LT}).
Using (\ref{eq:eta-recursion}), we find
\begin{equation*}
-\frac{\partial}{\partial w_i}\hatH_{0,2}(t_i,t_j)
\equiv
-\frac{1}{v_i}\;\frac{\partial}{\partial v_i}
\log\left(\eta_{-1}(v_i)-\eta_{-1}(v_j)\right)
\equiv
\frac{\eta_0(v_i)}{\eta_{-1}(v_i)-\eta_{-1}(v_j)}
\end{equation*}
modulo holomorphic functions in $w_i$ and $w_j$.
Therefore,  the result of the coordinate change from the
$t$-coordinates to the $v$-coordinates is the following:
\begin{multline}
\label{eq:caj-in-v}
\sum_{n_L}\la \tau_{n_L}\Lam_g ^\vee(1)\ra_{g,\ell}
\Bigg(
 (2g-2+\ell)\eta_{n_L}(v_L)
+
\sum_{i=1} ^\ell \eta_{n_i+1}(v_i)\eta_{L\setminus\{i\}}(
v_{L\setminus\{i\}})
\\
-
\sum_{i=1} ^\ell
\eta_{-1}(v_i)\eta_{n_i+1}(v_i)
\eta_{n_{L\setminus\{i\}}}(
v_{L\setminus\{i\}})
\Bigg)
\\
\equiv
\half\sum_{i=1}^\ell
\sum_{n_{L\setminus\{i\}}}\sum_{a,b}
\Bigg(
\la \tau_a\tau_b\tau_{n_{L\setminus\{i\}}}\Lam_{g-1} ^\vee
(1)\ra_{g-1,\ell+1}\\
+
\sum_{\substack{g_1+g_2 = g\\
I\sqcup J = L\setminus\{i\}}} ^{\rm{stable}}
\la \tau_a\tau_{n_{I}}\Lam_{g_1} ^\vee
(1)\ra_{g_1,|I|+1}
\la \tau_b\tau_{n_{J}}\Lam_{g_2} ^\vee
(1)\ra_{g_2,|J|+1}
\Bigg)
\eta_{a+1}(v_i)\eta_{b+1}(v_i)\eta_{n_{L\setminus\{i\}}}(
v_{L\setminus\{i\}})
\\
+
\half\sum_{i=1}^\ell\sum_{j\ne i}
\sum_{n_{L\setminus\{i,j\}}}\sum_{m }
\la \tau_{n_{L\setminus\{i,j\}}}\tau_{m }
\Lam_{g} ^\vee (1)\ra_{g,\ell-1}
\eta_{n_{L\setminus\{i,j\}}}(v_{L\setminus\{i,j\}})
\\
\times
 \frac{\eta_{m+1}(v_i)\eta_0(v_i)-\eta_{m+1}(v_j)\eta_0(v_j)}{\eta_{-1}(v_i)-\eta_{-1}(v_j)},
\end{multline}
again modulo terms containing any holomorphic
terms in any of $w_k$'s. 
At this stage we take the direct image
with respect to the projection $\pi:C\rightarrow \bC$
applied to the $v_1$-coordinate component,
and then restrict the result to its principal part,
meaning that we throw away any terms that contain
non-negative powers of any of the $v_k$'s.
Thanks to (\ref{eq:eta=xi+F}), only those terms
containing $\eta_a(v_1)\eta_b(v_1)$ survive. 
The last term of (\ref{eq:caj-in-v})
requires a separate care. We find
\begin{multline*}
\half\left(
\frac{\eta_{m+1}(v_1)\eta_0(v_1)-\eta_{m+1}(v_j)\eta_0(v_j)}{\eta_{-1}(v_1)-\eta_{-1}(v_j)}
+
\frac{\eta_{m+1}(v_1)\eta_0(v_1)-\eta_{m+1}(v_j)\eta_0(v_j)}{-\eta_{-1}(v_1)-\eta_{-1}(v_j)}
\right)
\\
=
-\left(
\eta_{m+1}(v_1)\eta_0(v_1)-\eta_{m+1}(v_j)\eta_0(v_j)
\right)
\frac{\eta_{-1}(v_j)}{\eta_{-1}(v_1)^2 -\eta_{-1}(v_j)^2}
\\
\equiv
\frac{\eta_{m+1}(v_j)}{v_1 ^2-v_j ^2}
=\frac{\eta_{m+1}(v_j)}{v_1 ^2}
 \sum_{k=0} ^\infty \left(\frac{v_j}{v_1}\right)^{2k},
\end{multline*}
modulo terms containing non-negative terms in $v_j$. 
Thus by taking the
direct image and reducing to the
principal part, (\ref{eq:caj-in-v}) is greatly simplified.
We have obtained:

\begin{thm}
\label{thm:ell=general}
\begin{multline}
\label{eq:caj-in-w}
-\sum_{n_L}\la \tau_{n_L}\Lam_g ^\vee(1)\ra_{g,\ell}
\eta_{-1}(v_1)\eta_{n_1+1}(v_1)
\eta_{n_{L\setminus\{1\}}}(
v_{L\setminus\{1\}})
\\
\equiv
\half
\sum_{n_{L\setminus\{1\}}}\sum_{a,b}
\Bigg(
\la \tau_a\tau_b\tau_{n_{L\setminus\{1\}}}\Lam_{g-1} ^\vee
(1)\ra_{g-1,\ell+1}\\
+
\sum_{\substack{g_1+g_2 = g\\
I\sqcup J = L\setminus\{1\}}} ^{\rm{stable}}
\la \tau_a\tau_{n_{I}}\Lam_{g_1} ^\vee
(1)\ra_{g_1,|I|+1}
\la \tau_b\tau_{n_{J}}\Lam_{g_2} ^\vee
(1)\ra_{g_2,|J|+1}
\Bigg)
\eta_{a+1}(v_1)\eta_{b+1}(v_1)\eta_{n_{L\setminus\{1\}}}(
v_{L\setminus\{1\}})
\\
+
\half\sum_{j\ge 1}
\sum_{n_{L\setminus\{1,j\}}}\sum_{m }
\la \tau_{n_{L\setminus\{1,j\}}}\tau_{m }
\Lam_{g} ^\vee (1)\ra_{g,\ell-1}
\eta_{n_{L\setminus\{1,j\}}}(v_{L\setminus\{1,j\}})
\frac{\eta_{m+1}(v_j)}{v_1 ^2}
 \sum_{k=0} ^{\rm{finite}} \left(\frac{v_j}{v_1}\right)^{2k}
\end{multline}
modulo terms with holomorphic factors in $v_k$.
\end{thm}

We note that only finitely many terms of the
expansion contributes in the last term of (\ref{eq:caj-in-w}).
Appealing to Corollary~\ref{cor:Pab} and
Proposition~\ref{prop:Pn}, we obtain
(\ref{eq:BMalgebraic}), after
switching back to the $t$-coordinates. 
 We have thus completed the proof of
 the Bouchard-Mari\~no conjecture \cite{BM}.

\setcounter{section}{0}
\setcounter{thm}{0}
\renewcommand{\thethm}{\thesection.\arabic{thm}}
\setcounter{equation}{0}
\renewcommand{\theequation}{\thesection.\arabic{equation}}
\setcounter{figure}{0}
\setcounter{table}{0}

\begin{appendix}
\section*{Appendix. Examples of linear Hodge integrals and 
Hurwitz numbers}
\renewcommand{\thesection}{A}

In this Appendix we give a few examples of linear
Hodge integrals and Hurwitz numbers 
 computed by Michael Reinhard.

\begin{table}[htb]
  \renewcommand\arraystretch{1.5}
  \centering

\begin{tabular}{|c||c|c||c|c||cc|}
\hline 
\multicolumn{7}{|c|}{$g=2$}\tabularnewline
\hline
\hline 
$\ell=1$ & $\langle\tau_{3}\lambda_{1}\rangle_{2,1}$ & $\frac{1}{480}$ & \multicolumn{1}{c}{} &  &  & \tabularnewline
\hline
\hline 
$\ell=2$ & $\langle\tau_{2}^{2}\lambda_{1}\rangle_{2,2}$ & $\frac{5}{576}$ & \multicolumn{1}{c}{} &  &  & \tabularnewline
\hline
\hline 
\multicolumn{7}{|c|}{$g=3$}\tabularnewline
\hline
\hline 
$\ell=1$ & $\langle\tau_{6}\lambda_{1}\rangle_{3,1}$ & $\frac{7}{138\,240}$ & $\langle\tau_{5}\lambda_{2}\rangle_{3,1}$ & $\frac{41}{580\,608}$ &  & \tabularnewline
\hline
\hline 
\multirow{2}{*}{$\ell=2$} & $\langle\tau_{2}\tau_{5}\lambda_{1}\rangle_{3,2}$ & $\frac{323}{483\,840}$ & $\langle\tau_{2}\tau_{4}\lambda_{2}\rangle_{3,2}$ & $\frac{2329}{2\,903\,040}$ &  & \tabularnewline
\cline{2-7} 
 & $\langle\tau_{3}\tau_{4}\lambda_{1}\rangle_{3,2}$ & $\frac{19}{17\,920}$ & $\langle\tau_{3}^{2}\lambda_{2}\rangle_{3,2}$ & $\frac{1501}{1\,451\,520}$ &  & \tabularnewline
\hline
\hline 
$\ell=3$ & $\langle\tau_{2}^{2}\tau_{4}\lambda_{1}\rangle_{3,3}$ & $\frac{541}{60\,480}$ & $\langle\tau_{2}\tau_{3}^{3}\lambda_{1}\rangle_{3,3}$ & $\frac{89}{7680}$ & \multicolumn{1}{c|}{$\langle\tau_{2}^{2}\tau_{3}\lambda_{2}\rangle_{3,3}$} & $\frac{859}{96\,768}$\tabularnewline
\hline
\hline 
$\ell=4$ & $\langle\tau_{2}^{3}\tau_{3}\lambda_{1}\rangle_{3,4}$ & $\frac{395}{3456}$ & $\langle\tau_{2}^{4}\lambda_{2}\rangle_{3,4}$ & $\frac{17}{192}$ &  & \tabularnewline
\hline
\hline 
\multicolumn{7}{|c|}{$g=4$}\tabularnewline
\hline
\hline 
$\ell=1$ & $\langle\tau_{9}\lambda_{1}\rangle_{4,1}$ & $\frac{1}{1\,244\,160}$ & $\langle\tau_{8}\lambda_{2}\rangle_{4,1}$ & $\frac{1357}{696\,729\,600}$ & \multicolumn{1}{c|}{$\langle\tau_{7}\lambda_{3}\rangle_{4,1}$} & $\frac{13}{6\,220\,800}$\tabularnewline
\hline
\hline 
\multirow{4}{*}{$\ell=2$} & $\langle\tau_{2}\tau_{8}\lambda_{1}\rangle_{4,2}$ & $\frac{841}{38\,707\,200}$ & $\langle\tau_{2}\tau_{7}\lambda_{2}\rangle_{4,2}$ & $\frac{33\,391}{696\,729\,600}$ & \multicolumn{1}{c|}{$\langle\tau_{3}\tau_{5}\lambda_{3}\rangle_{4,2}$} & $\frac{2609}{29\,030\,400}$\tabularnewline
\cline{2-7} 
 & $\langle\tau_{3}\tau_{7}\lambda_{1}\rangle_{4,2}$ & $\frac{221}{4\,147\,200}$ & $\langle\tau_{3}\tau_{6}\lambda_{2}\rangle_{4,2}$ & $\frac{1153}{11\,059\,200}$ & \multicolumn{1}{c|}{$\langle\tau_{4}^{2}\lambda_{3}\rangle_{4,2}$} & $\frac{6421}{58\,060\,800}$\tabularnewline
\cline{2-7} 
 & $\langle\tau_{4}\tau_{6}\lambda_{1}\rangle_{4,2}$ & $\frac{517}{5\,806\,080}$ & $\langle\tau_{4}\tau_{5}\lambda_{2}\rangle_{4,2}$ & $\frac{979}{6\,451\,200}$ &  & \tabularnewline
\cline{2-7} 
 & $\langle\tau_{5}^{2}\lambda_{1}\rangle_{4,2}$ & $\frac{1223}{11\,612\,160}$ & $\langle\tau_{2}\tau_{6}\lambda_{3}\rangle_{4,2}$ & $\frac{5477}{116\,121\,600}$ &  & \tabularnewline
\hline
\hline 
\multirow{4}{*}{$\ell=3$} & $\langle\tau_{2}^{2}\tau_{7}\lambda_{1}\rangle_{4,3}$ & $\frac{3487}{5\,806\,080}$ & $\langle\tau_{3}\tau_{4}^{2}\lambda_{1}\rangle_{4,3}$ & $\frac{137}{46\,080}$ & \multicolumn{1}{c|}{$\langle\tau_{3}^{2}\tau_{4}\lambda_{2}\rangle_{4,3}$} & $\frac{58\,951}{16\,588\,800}$\tabularnewline
\cline{2-7} 
 & $\langle\tau_{2}\tau_{3}\tau_{6}\lambda_{1}\rangle_{4,3}$ & $\frac{50\,243}{38\,707\,200}$ & $\langle\tau_{2}^{2}\tau_{6}\lambda_{2}\rangle_{4,3}$ & $\frac{137\,843}{116\,121\,600}$ & \multicolumn{1}{c|}{$\langle\tau_{2}^{2}\tau_{5}\lambda_{3}\rangle_{4,3}$} & $\frac{241}{230\,400}$\tabularnewline
\cline{2-7} 
 & $\langle\tau_{2}\tau_{4}\tau_{5}\lambda_{1}\rangle_{4,3}$ & $\frac{2597}{1\,382\,400}$ & $\langle\tau_{2}\tau_{3}\tau_{5}\lambda_{2}\rangle_{4,3}$ & $\frac{577}{258\,048}$ & \multicolumn{1}{c|}{$\langle\tau_{2}\tau_{3}\tau_{4}\lambda_{3}\rangle_{4,3}$} & $\frac{27\,821}{16\,588\,800}$\tabularnewline
\cline{2-7} 
 & $\langle\tau_{3}^{2}\tau_{5}\lambda_{1}\rangle_{4,3}$ & $\frac{3359}{1\,382\,400}$ & $\langle\tau_{2}\tau_{4}^{2}\lambda_{2}\rangle_{4,3}$ & $\frac{2657}{967\,680}$ & \multicolumn{1}{c|}{$\langle\tau_{3}^{3}\lambda_{3}\rangle_{4,3}$} & $\frac{4531}{2\,073\,600}$\tabularnewline
\hline
\hline 
\multicolumn{7}{|c|}{$g=5$}\tabularnewline
\hline
\hline 
\multirow{2}{*}{$\ell=1$} & $\langle\tau_{12}\lambda_{1}\rangle_{5,1}$ & $\frac{1}{106\,168\,320}$ & $\langle\tau_{10}\lambda_{3}\rangle_{5,1}$ & $\frac{71}{1\,114\,767\,360}$ &  & \tabularnewline
\cline{2-7} 
 & $\langle\tau_{11}\lambda_{2}\rangle_{5,1}$ & $\frac{577}{16\,721\,510\,400}$ & $\langle\tau_{9}\lambda_{4}\rangle_{5,1}$ & $\frac{21\,481}{367\,873\,228\,800}$ &  & \tabularnewline
\hline
\end{tabular}

\bigskip
  
  \caption{Examples of linear Hodge integrals.}
\end{table}

Some examples of $g = 5$ Hurwitz numbers:
\begin{align*}
h_{5,(1)} &= 0 & h_{5,(4)} &= 272\,097\,280 \\
h_{5,(2)} &= 1/2 & h_{5,(5)} &= 333\,251\,953\,125 \\
h_{5,(3)} &= 59\,049 & h_{5,(6)} &= 202\,252\,053\,177\,720 \\
\end{align*}

\begin{table}[htb]
  \centering
  
  \begin{tabular}{|c||c|c|c|c|}
\hline 
$h_{g,\mu}$ & $g=1$ & $g=2$ & $g=3$ & $g=4$\tabularnewline
\hline
\hline 
$(1)$ & $0$ & $0$ & $0$ & $0$\tabularnewline
\hline 
$(2)$ & $1/2$ & $1/2$ & $1/2$ & $1/2$\tabularnewline
\hline 
$(1,1)$ & $1/2$ & $1/2$ & $1/2$ & $1/2$\tabularnewline
\hline 
$(3)$ & $9$ & $81$ & $729$ & $6561$\tabularnewline
\hline 
$(2,1)$ & $40$ & $364$ & $3280$ & $29\,524$\tabularnewline
\hline 
$(1,1,1)$ & $40$ & $364$ & $3280$ & $29\,524$\tabularnewline
\hline 
$(4)$ & $160$ & $5824$ & $209\,920$ & $7\,558\,144$\tabularnewline
\hline 
$(3,1)$ & $1215$ & $45\,927$ & $1\,673\,055$ & $60\,407\,127$\tabularnewline
\hline 
$(2,2)$ & $480$ & $17\,472$ & $629\,760$ & $22\,674\,432$\tabularnewline
\hline 
$(2,1,1)$ & $5460$ & $206\,640$ & $7\,528\,620$ & $271\,831\,560$\tabularnewline
\hline 
$(1,1,1,1)$ & $5460$ & $206\,640$ & $7\,528\,620$ & $ $\tabularnewline
\hline 
$(5)$ & $3125$ & $328\,125$ & $33\,203\,125$ & $3\,330\,078\,125$\tabularnewline
\hline 
$(4,1)$ & $35\,840$ & $3\,956\,736$ & $409\,108\,480$ & $41\,394\,569\,216$\tabularnewline
\hline 
$(3,2)$ & $26\,460$ & $2\,748\,816$ & $277\,118\,820$ & $27\,762\,350\,616$\tabularnewline
\hline 
$(3,1,1)$ & $234\,360$ & $26\,184\,060$ & $2\,719\,617\,120$ & $275\,661\,886\,500$\tabularnewline
\hline 
$(2,2,1)$ & $188\,160$ & $20\,160\,000$ & $2\,059\,960\,320$ & $207\,505\,858\,560$\tabularnewline
\hline 
$(2,1,1,1)$ & $1\,189\,440$ & $131\,670\,000$ & $13\,626\,893\,280$ & $ $\tabularnewline
\hline 
$(1,1,1,1,1)$ & $1\,189\,440$ & $131\,670\,000$ & $ $ & $ $\tabularnewline
\hline 
$(6)$ & $68\,040$ & $16\,901\,136$ & $3\,931\,876\,080$ & $895\,132\,294\,056$\tabularnewline
\hline 
$(5,1)$ & $1\,093\,750$ & $287\,109\,375$ & $68\,750\,000\,000$ & $15\,885\,009\,765\,625$\tabularnewline
\hline 
$(4,2)$ & $788\,480$ & $192\,783\,360$ & $44\,490\,434\,560$ & $10\,093\,234\,511\,360$\tabularnewline
\hline 
$(4,1,1)$ & $9\,838\,080$ & $2\,638\,056\,960$ & $638\,265\,788\,160$ & $148\,222\,087\,453\,440$\tabularnewline
\hline 
$(3,3)$ & $357\,210$ & $86\,113\,125$ & $19\,797\,948\,720$ & $4\,487\,187\,539\,835$\tabularnewline
\hline 
$(3,2,1)$ & $14\,696\,640$ & $3\,710\,765\,520$ & $872\,470\,478\,880$ & $199\,914\,163\,328\,880$\tabularnewline
\hline 
$(3,1,1,1)$ & $65\,998\,800$ & $17\,634\,743\,280$ & $4\,259\,736\,280\,800$ & $ $\tabularnewline
\hline 
$(2,2,2)$ & $2\,016\,000$ & $486\,541\,440$ & $111\,644\,332\,800$ & $25\,269\,270\,586\,560$\tabularnewline
\hline 
$(2,2,1,1)$ & $80\,438\,400$ & $20\,589\,085\,440$ & $4\,874\,762\,692\,800$ & $ $\tabularnewline
\hline 
$(2,1,1,1,1)$ & $382\,536\,000$ & $100\,557\,737\,280$ & $ $ & $ $\tabularnewline
\hline 
$(1,1,1,1,1,1)$ & $382\,536\,000$ & $ $ & $ $ & $ $\tabularnewline
\hline
\end{tabular}
\bigskip

  \caption{Examples of Hurwitz numbers for $1 \leq g \leq 4$
  and  $|\mu| \leq 6$.}
\end{table}

\end{appendix}

%Bibliography

\providecommand{\bysame}{\leavevmode\hbox to3em{\hrulefill}\thinspace}

\bibliographystyle{amsplain}

\end{document}